\documentclass[12pt,reqno]{amsproc}
\usepackage[all]{xy}
\usepackage{mathrsfs, amsfonts, amsthm, amsmath, amssymb}
\usepackage{rotating}
\usepackage{graphicx}
\usepackage{hyperref}
\usepackage{enumitem}

\usepackage[left]{lineno}

\hypersetup{
  colorlinks   =  true, 
  urlcolor     = black, 
  linkcolor    = black, 
  citecolor   = black 
}

 \theoremstyle{plain}

\theoremstyle{plain}
\newtheorem{theorem}{Theorem}[section]

\newtheorem{definition}[theorem]{Definition}

\newtheorem{lemma}[theorem]{Lemma}

\newtheorem{remark}[theorem]{Remark}

\def\A{\mathcal A}
\def\B{\mathcal B}

\def\F{\mathcal F}

\def\H{\mathcal H}

\def\K{\mathcal K}

\def\M{\mathcal M}
\def\N{\mathcal N}

\def\P{\mathcal P}

\def\R{\mathcal R}

\def\CCC{\mathbb C}
\def\FFF{\mathbb F}
\def\NNN{\mathbb N}

\newtheorem*{Voicu}{\bf Voiculescu's Theorem}
\newtheorem*{Hadwin}{\bf Hadwin's Theorem}

\makeatletter
\newcommand{\LeftEqNo}{\let\veqno\@@leqno}

\makeatother

 \numberwithin{equation}  {section}
\begin{document}

\

\vspace{-2cm}


\title[Approximate equivalence of representations of AH algebras]{Approximate equivalence of representations of AH algebras into semifinite von Neumann factors}

\author{Junhao Shen}
\address{Department of Mathematics \& Statistics, University of
New Hampshire, Durham, 03824, US}
\email{Junhao.Shen@unh.edu}
\thanks{}

\author{Rui Shi}
\address{School of Mathematical Sciences, Dalian University of
Technology, Dalian, 116024, China}
\email{ruishi@dlut.edu.cn, ruishi.math@gmail.com}




\begin{abstract}
In this paper, we prove a non-commutative version of the Weyl-von Neumann theorem for representations of unital, separable AH algebras into countably decomposable, semifinite, properly infinite, von Neumann factors, where an AH algebra means an approximately homogeneous ${\rm C}^{\ast}$-algebra. We also prove a result for approximate summands of representations of unital, separable AH algebras into finite von Neumann factors.
\end{abstract}

\maketitle

\section{Introduction}

The classical Weyl-von Neumann theorem states that, for each bounded linear self-adjoint operator $a$ on a separable Hilbert space $\mathcal{H}$, there is a diagonal self-adjoint operator $d$ such that $a-d$ is a Hilbert-Schmidt operator of arbitrarily small norm, which was first proved by Weyl \cite{Weyl} in 1909 and was improved by von Neumann \cite{Von2} in 1935. This theorem provides important techniques in the perturbation theory for bounded linear operators on $\H$.  

In 1971, a version of Weyl-von Neumann theorem for normal operators was proved in \cite{Berg}. It states that a normal operator $a$ is diagonalizable up to an arbitrarily small compact perturbation. As a corollary of the main theorem in \cite{Voi}, Voiculescu proved that a normal operator is diagonalizable up to an arbitrarily small Hilbert-Schimidt perturbation. Later, another proof was provided by Davidson in \cite{DavidsonNormal}. 

As one important ingredient of the famous Brown-Douglas-Fillmore theory, the Weyl-von Neumann theorem attracted much attention during the past decades. While Voiculescu \cite{Voi2} proved the striking non-commutative Weyl-von Neumann theorem in $\B(\H)$, several interesting commutative versions of the Weyl-von Neumann theorem are proved in the setting of semifinite von Neumann algebras (\cite{Zaido, Kaftal, Li, Hadwin3}). 

One goal of the current paper is to set up a non-commmutative version of the Weyl-von Neumann theorem in semifinite factor von Neumann algebras. For this purpose, we have to face two main obstacles. One is the fact that a semifinite von Neumann algebra might contain no minimal projections, since minimal projections play an important role in the proof of the non-commmutative Weyl-von Neumann theorem in the case of $\B(\H)$. The other obstacle is that the non-commutative AH algebras are much more complicated than commutative ${\rm C}^{\ast}$-algebras. To deal with these two obstructions, we develop new techniques in semifinite von Neumann algebras.

We recall several terms in von Neumann algebras. A \emph{von Neumann algebra} is a $*$-algebra of bounded linear operators on a Hilbert space which is closed in the weak operator topology and contains the identity. A \emph{ factor} (or \emph{von Neumann factor}) is a von Neumann algebra with trivial center. Factors are classified by Murray and von Neumann \cite{Murray} into three types, i.e., type I, type II, and type III factors. A factor is called \emph{semifinite} if it is of type I or II. This is  equivalent to say that a factor equipped with a faithful, normal, semifinite, tracial weight, is semifinite (see Definition 7.5.1 of \cite{Kadison2} for a weight on a ${\rm C}^{\ast}$-algebra).   A factor is called \emph{(properly) infinite} if the identity is an infinite projection. The reader is referred to \cite{Kadison1, Kadison2, Takesaki, Takesaki3, Blackadar} for the theory of von Neumann algebras.  Throughout this paper, let $\H$ be a complex, separable Hilbert space and $\B(\H)$ be the set of all the bounded linear operators on $\H$. By definition, $\B(\H)$ is a factor of type I.

In the setting of semifinite von Neumann factors, Kaftal \cite{Kaftal} and Zsid\'{o} \cite{Zaido} proved the extended Weyl-von Neumann theorem for self-adjoint operators. Inspired by these perturbation results in semifinite von Neumann factors,  the authors of \cite{Li} proved a Weyl-von Neumann theorem for normal operators with respect to an arbitrarily small $\max\{\Vert\cdot\Vert, \Vert\cdot\Vert_{2}\}$-perturbation in  semifinite von Neumann factors $(\M, \tau)$, where by $\Vert\cdot\Vert$ we denote the operator norm and the norm $\Vert \cdot\Vert_2$ is defined as $\Vert x\Vert_2:=\tau(|x|^2)^{1/2}$ for every $x\in\M$ (see Theorem 6.1.2 of \cite{Li}). 

In a different point of view, the Weyl-von Neumann-Berg theorem states that, for a unital, separable, commutative ${\rm C}^{\ast}$-subalgebra $\A$ of $\B(\H)$, every $\ast$-representation of $\A$ is approximately unitarily equivalent to a diagonal representation relative to $\K(\H)$, where $\K(\H)$ denotes the two-sided closed ideal of all the compact operators in $\B(\H)$ (see Theorem II.4.6 of \cite{Davidson} for more details).  

It is remarkable that Voiculescu proved the non-commutative Weyl-von Neumann theorem in  \cite{Voi2} for every unital separable ${\rm C}^{\ast}$-subalgebra of $\B(\H)$ instead of commutative ones. In the current paper, we will refer to the theorem in \cite{Voi2} as Voiculescu's theorem for simplicity. Precisely, Voiculescu's theorem is cited as follows:

\begin{Voicu}[\cite{Voi2}]\label{Voiculescu}
	Suppose $\mathcal{A}$ is a separable unital ${\rm C}^{\ast}$-algebra and $\mathcal{H}$ is a complex, separable Hilbert space. Let $\phi$ and $\psi$ be unital $\ast$-representations of $\A$ into $\B(\H)$. The following statements are equivalent:
	\begin{enumerate}
		\item $\phi \sim_{a}\psi$.
		\item $\phi \sim_{\mathcal{A}}\psi \mod \mathcal{K}(\mathcal{H})$.
		\item $\ker \phi=\ker \psi$, $\phi^{-1}\left(  \mathcal{K}(\mathcal{H})\right)  =\psi^{-1}\left(  \mathcal{K}(\mathcal{H})\right)$, and the nonzero parts of the restrictions $\phi|_{\phi^{-1}\left(\mathcal{K}(\mathcal{H})\right)}$ and $\psi|_{\psi^{-1}\left(\mathcal{K}(\mathcal{H})\right)}$ are unitarily equivalent.
	\end{enumerate}
\end{Voicu}

In this theorem, by $\phi \sim_{a}\psi$, we mean the approximately unitary equivalence of $\phi$ and $\psi$, i.e., there exists a sequence of unitary operators $\{u_{n}\}^{\infty}_{n=1}$ in $\B(\H)$ such that
\begin{equation*}
	\lim_{n\to\infty}\Vert u^*_{n}\phi(a)u_{n}-\psi(a)\Vert=0,\quad \forall \ a\in \A.
\end{equation*}
By $\phi \sim_{\mathcal{A}}\psi \mod \mathcal{K}(\mathcal{H})$, we mean the approximately unitary equivalence of $\phi$ and $\psi$ relative to $\K(\H)$, i.e., there exists a sequence of unitary operators $\{u_{n}\}^{\infty}_{n=1}$ in $\B(\H)$ such that
\begin{enumerate}
	\item for each $a$ in $\A$ and every $n\in\NNN$,
	\begin{equation*}
	u^*_{n}\phi(a)u_{n}-\psi(a)\in\K(\H),\ \text{and}
\end{equation*}
	\item for each $a$ in $\A$,
	\begin{equation*}
	\lim_{n\to\infty}\Vert u^*_{n}\phi(a)u_{n}-\psi(a) \Vert=0.
\end{equation*}
\end{enumerate}

This theorem is important in both operator theory and operator algebras. Its applications can be found in the diagonalization problem of normal operators in \cite{Voi}, the eighth problem proposed by Halmos in \cite{Halmos}, and the well-known Brown-Douglas-Fillmore theory (see Chapter IX of \cite{Davidson}).

By introducing quasi-central approximate units of ${\rm C}^{\ast}$-algebras, Arveson  provided another beautiful proof of Voiculescu's theorem in \cite{Ave}. Later, Hadwin \cite{Hadwin1} provided an algebraic characterization of approximate equivalence of representation. In \cite{Hadwin1}, Hadwin also proved the following result:

\begin{Hadwin}[Lemma 2.3 of \cite{Hadwin1}]
	Suppose that $\mathcal{A}$ is a separable unital ${\rm C}^{\ast}$-algebra, $\mathcal{H}_0$ and $\mathcal{H}_1$ are Hilbert spaces, Let $\phi:\mathcal{A\rightarrow}\mathcal{B}(\mathcal{H}_0)$ and $\psi:\mathcal{A\rightarrow}\mathcal{B}(\mathcal{H}_1)$ be unital representations. The following are equivalent:
	\begin{enumerate}
		\item There is a representation $\gamma:\mathcal{A}\rightarrow \mathcal{B}(\mathcal{H}_2)$ for some Hilbert space $\mathcal{H}_{2}$ such that
		\begin{equation*}
			\psi \oplus \gamma \sim_{a}\phi. 
		\end{equation*}
		\item For every $A\in \mathcal{A}$,
		\begin{equation*}
			\mathrm{rank}\left(\psi\left(A\right)\right)\leq\mathrm{rank}\left(\phi \left(A\right)\right). 
		\end{equation*}
	\end{enumerate}
\end{Hadwin}

In \cite{Hadwin2}, Ding and Hadwin extended some results of \cite{Hadwin1} to the case where $\mathcal{B}(\mathcal{H}_0)$ is replaced with a von Neumann algebra.

As another important application of \cite{Voi2}, the authors of \cite{Niu} characterized properly infinite injective von Neumann algebras and nuclear ${\rm C}^{\ast}$-algebras by using a uniqueness theorem.

Inspired by the preceding interesting results, we focus on analogues of Voiculescu's theorem and Hadwin's theorem in the setting of semifinite von Neumann factors. 

This paper is organized as follows. Since factors of type II contain no minimal projections, they are quite different from factors of type I. Thus, to prove the main theorems in the current paper, we need to prepare related notation and definitions in Section 2. In particular, we introduce the  strongly-approximately-unitarily-equivalent $\ast$-homomorphisms which was first defined in \cite{Li} to extend the concept of approximately unitarily equivalence of $\ast$-homomorphisms relative to $\K(\H)$ in the setting of $\B(\H)$. 

Note that, at the end of \cite{Hadwin2}, Hadwin pointed out that there exist unital representations $\pi$ and $\rho$  of  ${\rm C}^{\ast}(\FFF_2)$ into a hyperfinite type II$_1$ factor $(\R,\tau)$ with $\tau\circ \pi=\tau\circ \rho$ such that $\pi$ and $\rho$ are not weakly approximately equivalent in $\R$. In this sense, we can't expect to extend Voiculescu's theorem for every ${\rm C}^{\ast}$-subalgebra in semifinite von Neumann factors. 

In terms of \cite{Hadwin3}, it is reasonable to choose the family of AH-algebras, a classical collection of inductive limit ${\rm C}^{\ast}$-algebras, to build up a non-commutative version of Voiculescu's theorem in semifinite von Neumann factors.  Meanwhile, the choice makes sense, since the family of AH algebras plays an important role in the study of ${\rm C}^{\ast}$-algebras (see \cite{Gong3,Gong1,Lin1,Lin2,Lin3,Niu1}). Hence, we prepare the definition of AH algebras and certain properties of AH algebras from \cite{Ror, Blackadar} in Section 2. 

In Section 3, we prove the following theorem for AH algebras in ${\rm II}_{1}$ factors:

\vspace{0.2cm} {T{\scriptsize HEOREM}} \ref{thm-3.9}.
\emph{Let $\mathcal{A}$ be a unital separable AH subalgebra in a type ${\rm II}_1$ factor $(\mathcal{N},\tau)$ with separable predual. Let $p$ be a projection in $(\mathcal{N},\tau)$. Suppose that $\pi:\mathcal{A} \rightarrow \mathcal{N}$ is a unital $\ast$-homomorphism and $\rho: \mathcal{A} \rightarrow p\mathcal{N}p$ is a unital $\ast$-homomorphism such
	that:	
	\begin{equation*}
		\tau(  R\left(\rho \left(  a\right)\right))\leq \tau(R\left(\pi\left( a \right)\right)), \quad\forall \ a \in \mathcal{A}. 
	\end{equation*}
	Then, there exists a unital $\ast$-homomorphism $\gamma: \mathcal{A} \rightarrow p^{\perp}\mathcal{N} p^{\perp}$ such that
	\begin{equation*}
		\rho\oplus\gamma\sim_{a}\pi \quad\mbox{ in }\ \mathcal{N}.
	\end{equation*}
	}
	
In Theorem  \ref{thm-3.9}, by $\rho\oplus\gamma\sim_{a}\pi$ in $\mathcal{N}$, we mean the approximately unitary equivalence of $\pi$ and $\rho\oplus\gamma$ in $\N$, i.e., there exists a sequence of unitary operators $\{u_{n}\}^{\infty}_{n=1}$ in $\N$ such that
\begin{equation*}
	\lim_{n\to\infty}\Vert u^*_{n}\big((\rho\oplus\gamma)(a)\big)u_{n}-\pi(a)\Vert=0,\quad \forall \ a\in \A.
\end{equation*}

In Section 4, we prove an extended Voiculescu's theorem for AH algebras in semifinite, (properly) infinite von Neumann factors:

\vspace{0.2cm} {T{\scriptsize HEOREM}} \ref{AH-main-thm}.
\emph{Let $\mathcal{M}$ be a countably decomposable, properly infinite, semifinite factor with a faithful, normal, semifinite, tracial weight $\tau$. Suppose that $\mathcal{A}$ is a separable AH subalgebra of $\mathcal{M}$ with an identity $I_{\mathcal{A}}$. }

\emph{If $\phi$ and $\psi$ are unital $\ast$-homomorphisms of $\mathcal{A}$ into $\mathcal{M}$, then the following statements are equivalent:
	\begin{enumerate}
		\item[(i)] $\phi \sim_{a}\psi$ \ in $\mathcal{M}$;
		\item[(ii)] $\phi \sim_{\mathcal{A}}\psi \mod\ \mathcal{K}(\mathcal{M},\tau)$.
	\end{enumerate}}
	
The reader is referred to Definition \ref{AEC} in Section 2 for the notation $\phi \sim_{\mathcal{A}}\psi \mod\ \mathcal{K} (\mathcal{M},\tau)$.

\section{Preliminary}

In the following definition, the definitions of finite rank operators and compact operators in $\B(\H)$ are extended to analogues in von Neumann algebras with a faithful, normal, semifinite, tracial weight $\tau$. These definitions  will be frequently mentioned in this paper.

\begin{definition}\label{M-rank-opts}
	Let $(\mathcal{M},\tau)$ be a von Neumann algebra with a faithful, normal, semifinite, tracial weight $\tau$, where $\M$ acts on a Hilbert space $\H$. Let $\Vert\cdot\Vert$ denote the operator norm. 
	
	Define 
	\begin{equation}\label{compact-ideal-R-tau} 
		\mathcal P\mathcal F(\mathcal{M},\tau) =\{ p  \ : \ p=p^*=p^2\in \mathcal{M} \text { and } \tau(p)<\infty\}
	\end{equation}
	to be the set of finite trace projections in $(\mathcal{M},\tau)$. In terms of $\mathcal P\mathcal F(\mathcal{M},\tau)$, define $\mathcal F(\mathcal{M},\tau)$ to be the set in the form
	\begin{equation}\label{F(M,tau)}
		\mathcal F(\mathcal{M},\tau) = \{xpy \ : \ p\in  \mathcal P\mathcal F(\mathcal{M},\tau)  \text { and } x, y\in\mathcal{M}\}.
	\end{equation}
	Each element in $\mathcal F(\mathcal{M},\tau)$ is said to be of $(\M,\tau)$-finite-rank. When no confusion can arise, elements in $\F(\M,\tau)$ are called finite-rank operators. If $\M$ is of type $I$, then $\F(\M,\tau)$ coincides with the set of finite rank operators. In this point of view, the concept of finite-rank operator in $(\M,\tau)$ is a natural analogue of the concept of finite rank operator in $\B(\H)$.

	Define $\mathcal K(\mathcal{M},\tau)$ to be the $\|\cdot\|$-norm closure of $\mathcal F(\mathcal{M},\tau)$ in $\M$. Each element in $\mathcal K(\mathcal{M},\tau)$ is said to be \textbf{compact in} $\M$.
	
	For an element $x\in\mathcal{M}$, denote by $R(x)$ the range projection of $x$. From Proposition $6.1.6$ of \cite{Kadison2}, an operator $a$ in $\M$ is of  $(\M,\tau)$-finite-rank if and only if $\tau(R(a))<\infty$. 
\end{definition}

\begin{remark}
	By virtue of Theorem $6.8.3$ of \cite{Kadison2}, $\mathcal K(\mathcal{M},\tau)$ is a $\Vert\cdot\Vert$-norm closed two-sided ideal in $\M$. We denote by $\mathcal{K}(\mathcal{M})$ the $\| \cdot\|$-norm closed ideal generated by finite projections in $\mathcal{M}$. In general, $\mathcal{K}(\mathcal{M},\tau)$ is a subset of $\mathcal{K}(\mathcal{M})$. That is because a finite projection might not be a finite-rank projection with respect to $\tau$. However, if $\mathcal{M}$ is a countably decomposable, semifinite factor, then Proposition $8.5.2$ entails that
	\[
	\mathcal{K}(\mathcal{M},\tau) = \mathcal{K}(\mathcal{M})
	\]
	for a faithful, normal, semifinite tracial weight $\tau$.
\end{remark}

To introduce the definition of approximate equivalence of two unital $\ast$-homomorphisms of a separable C$^{\ast}$-algebra $\mathcal{A}$ into $\mathcal{M}$ (\textbf{relative to $\mathcal{K}(\mathcal{M},\tau)$}), we need to develop the following notation and definitions.

Suppose that $\{e_{i,j}\}_{i,j= 1}^\infty$ is a system of matrix units for $\mathcal{B}({l^2})$. For a countably decomposable, properly infinite von Neumann algebra $\mathcal{M}$ with a faithful normal semifinite tracial weight $\tau $, there exists a sequence $\{v_i\}_{i = 1}^\infty$ of partial isometries in $\mathcal{M}$ such that
\begin{equation}\label{v_i}
v_iv_i^*=I_{\mathcal{M}}, \ \ \ \ \sum_{i = 1}^\infty v_i^*v_i=I_{\mathcal{M}}, \ \ \ \ \text{ and } v_jv_i^*=0 \text { when } i\ne j.
\end{equation}

\begin{definition}\label{vN-alg-tensor}
For all $x\in \mathcal{M }$ and all \ $\sum_{i,j= 1}^\infty
x_{i,j}\otimes e_{i,j}\in \mathcal{M}\otimes \mathcal{B}({l^2})$,
define
\begin{equation*}
\phi: \mathcal{M}\rightarrow \mathcal{M}\otimes \mathcal{B}({l^2}) \ \
\ \text{ and } \ \ \psi: \mathcal{M}\otimes \mathcal{B}({l^2})
\rightarrow \mathcal{M}
\end{equation*}
by
\begin{equation*}
\phi(x) =\sum_{i,j= 1}^\infty (v_ixv^*_j)\otimes e_{i,j} \ \ \ \ \text{ and }
\ \ \ \ \psi( \sum_{i,j= 1}^\infty x_{i,j}\otimes e_{i,j} )= \sum_{i,j=1}^\infty v_i^*x_{i,j}v_j.
\end{equation*}
where $\{v_i\}_{i = 1}^\infty$ is a sequence of partial isometries in $\mathcal{M}$ as in $(\ref{v_i})$ and $\{e_{i,j}\}_{i,j= 1}^\infty$ is a system of matrix units for $\mathcal{B}(l^2)$ such that $\sum^{\infty}_{i=1}e_{i,i}$ equals the identity of $\mathcal{B}(l^2)$.

We further define a mapping $\tilde \tau :(\mathcal{M}\otimes \mathcal{B}({l^2}))_{+}\rightarrow
[0,\infty]$ to be
\begin{equation*}
\tilde \tau (y)=\tau(\psi(y)), \qquad \forall \ y\in
(\mathcal{M}\otimes \mathcal{B}({l^2}))_{+}.
\end{equation*}
\end{definition}

By Lemma 2.2.2 of \cite{Li}, both $\phi$ and $\psi$ are normal $*$-homomorphisms
satisfying
\begin{equation*}
\text{ $\psi \circ \phi=id_{\mathcal{M}}$ \qquad and \qquad $\phi \circ
\psi=id_{\mathcal{M}\otimes \mathcal{B}({l^2})}.$}
\end{equation*}

The following statements are proved in Lemma 2.2.4 of \cite{Li}:
\begin{enumerate}
\item[(i)] $\displaystyle\tilde \tau$ is a faithful, normal, semifinite tracial weight
of $\mathcal{M}\otimes \mathcal{B}({l^2})$.

\item[(ii)] $\displaystyle\tilde \tau(\sum_{i,j=1}^\infty x_{i,j}\otimes
e_{i,j})=\sum_{i=1}^\infty \tau(x_{i,i})$ for all \ $\displaystyle\sum_{i,j=1}^\infty x_{i,j}\otimes e_{i,j}\in (\mathcal{M}\otimes \mathcal{B}({l^2}))_{+}$.

\item[(iii)]
\begin{equation*}
\begin{aligned}
\mathcal P\mathcal F(\mathcal{M}\otimes \mathcal B({l^2}), \tilde \tau)&=\phi(\mathcal P\mathcal F(\mathcal{M},\tau)), \quad \\
\mathcal F(\mathcal{M}\otimes \mathcal B({l^2}), \tilde \tau)&=\phi(\mathcal F(\mathcal{M},\tau)), \quad \\
\mathcal K(\mathcal{M}\otimes \mathcal B({l^2}), \tilde \tau)&=\phi(\mathcal K(\mathcal{M},\tau)).
\end{aligned}
\end{equation*}
\end{enumerate}

\begin{remark}\label{tau-tilde}
Note that $\tilde \tau$ is a natural extension of $\tau$ from $\mathcal{M}$ to $\mathcal{M}\otimes \mathcal{B}({l^2})$. If no confusion arises, $\tilde\tau$ will be also denoted by $\tau$. By Proposition $2.2.9$ of \cite{Li}, the ideal $\mathcal K(\mathcal{M}\otimes \mathcal B({l^2}), \tilde \tau)$ is independent of the choice of the system of matrix units $\{e_{i,j}\}_{i,j=1}^\infty$  of $\mathcal{B}(l^2)$ and the choice of the family $\{v_i\}_{i =1}^\infty$ of partial isometries in $\mathcal{M}$.
\end{remark}

Now we are ready to introduce the definition of approximate equivalence of $\ast$-homomorphisms of a separable C$^{\ast}$-algebra into $\mathcal{M}$ relative to $\mathcal{K}_{}(\mathcal{M},\tau)$.

Let $\mathcal{A}$ be a separable ${\rm C}^*$-subalgebra of $\mathcal{M}$ with an identity $I_{\mathcal{A}}$.  Suppose that $\rho$ is a positive mapping from $\mathcal{A}$ into $\mathcal{M}$ such that $\rho(I_{\mathcal{A}})$ is a projection in $\mathcal{M}$. Then for all $0\le x\in \mathcal{A}$, we have $0\le \rho(x) \le \|x\|\rho(I_{\mathcal{A}})$. Therefore, it follows that
\[
\rho(x)\rho(I_{\mathcal{A}})=\rho(I_{\mathcal{A}})\rho(x)=\rho(x)
\]
for all positive $x\in \mathcal{A}$. In other words, $\psi(I_{\mathcal{A}})$ can be viewed as an identity of $\psi(\mathcal{A})$.
Or, $\psi(\mathcal{A})\subseteq \psi(I_{\mathcal{A}})\mathcal{M}\psi(I_{\mathcal{A}})$.

The following definition is a special case of Definition $2.3.1$ of \cite{Li} when the norm is fixed to be the operator norm $\Vert\cdot \Vert$.

\begin{definition}\label{AEC}
Let $\A$ be a separable ${\rm C}^{\ast}$-subalgebra of $\mathcal{M}$ with an identity $I_{\mathcal{A}}$ and $\mathcal{B}$ a $*$-subalgebra of $\mathcal{A}$ such that $I_{\mathcal{A}}\in\mathcal{B}$. Suppose that $\{e_{i,j}\}_{i,j\ge 1}$ is a system of matrix units for $\mathcal{B}({l^2})$. Let $M,N\in \mathbb{N}\cup\{\infty\}$. Suppose that $\psi_1,\ldots, \psi_M$ and $\phi_1,\ldots, \phi_N$ are positive mappings from $\mathcal{A}$ into $\mathcal{M}$ such
that $\psi_1(I_{\mathcal{A}}),\ldots, \psi_M(I_{\mathcal{A}})$, $\phi_1(I_{\mathcal{A}}),\ldots, \phi_N(I_{\mathcal{A}})$ are projections in $\mathcal{M}$.

\begin{enumerate}
\item[(a)] Let $\mathcal{F}\subseteq \mathcal{A}$ be a finite subset and $\epsilon>0$. Then we say that
\begin{equation*}
\text{ $\psi_1\oplus \cdots \oplus\psi_M$ is \emph{$(\mathcal{F},\epsilon)$-strongly-approximately-unitarily-equivalent} to $\phi_1\oplus\cdots\oplus\phi_N$ over $\B$,}
\end{equation*}
denoted by
\begin{equation*}
\psi_1\oplus\psi_2\oplus\cdots \oplus\psi_M \sim_{\B}^{(\mathcal{F},\epsilon)} \phi_1\oplus \phi_2\oplus\cdots\oplus \phi_N, \quad \mod \mathcal{K}_{}(\mathcal{M},\tau) 
\end{equation*}
if there exists a partial isometry $v$ in $\mathcal{M}\otimes \mathcal{B}({l^2})$ such that

\begin{enumerate}
\item[(i)] $\displaystyle v^*v= \sum_{i=1}^M \psi_i(I_{\mathcal{A}})\otimes e_{i,i}$ and $\displaystyle vv^*= \sum_{i=1}^N \phi_i(I_{\mathcal{A}})\otimes e_{i,i}$;

\item[(ii)] $\displaystyle  \sum_{i=1}^M \psi_i(x)\otimes e_{i,i} - v^*\left(\sum_{i=1}^N \phi_i(x)\otimes e_{i,i}\right ) v \in \mathcal{K}(\mathcal{M}\otimes \mathcal{B}({l^2}),\tau)$ for all $x\in \B$;

\item[(iii)] $\displaystyle \|\sum_{i=1}^M \psi_i(x)\otimes
e_{i,i} - v^*\left (\sum_{i=1}^N \phi_i(x)\otimes e_{i,i}\right ) v
\|<\epsilon$ for all $x\in \mathcal{F}$.
\end{enumerate}

\item[(b)] We say that
\begin{equation*}
\text{ $\psi_1\oplus \cdots \oplus\psi_M$ is \emph{strongly-approximately-unitarily-equivalent } to $\phi_1\oplus\cdots\oplus\phi_N$ over $\B$,}
\end{equation*}
denoted by
\begin{equation*}
\psi_1\oplus\psi_2\oplus\cdots \oplus\psi_M \sim_{\B} \phi_1\oplus
\phi_2\oplus\cdots\oplus \phi_N, \qquad \mod \mathcal{K}_{}(\mathcal{M},\tau)
\end{equation*}
if, for any finite subset $\mathcal{F}\subseteq \B$ and $\epsilon>0$,
\begin{equation*}
\psi_1\oplus\psi_2\oplus\cdots \oplus\psi_M \sim_{\B}^{(\mathcal{F},\epsilon)} \phi_1\oplus \phi_2\oplus\cdots\oplus \phi_N, \qquad \mod \mathcal{K}_{}(\mathcal{M},\tau).
\end{equation*}
\end{enumerate}
\end{definition}

By virtue of the preceding definitions, assume that $\mathcal{M}=\mathcal{B}(\mathcal{H})$ for a complex, separable, infinite dimensional, Hilbert space $\mathcal{H}$, $\phi$ and $\psi$ are unital $\ast$-homomorphisms of $\mathcal{A}$ into $\mathcal{M}$. It follows that $\phi$ and $\psi$ are strongly-approximately-unitarily-equivalent over $\mathcal{A}$ if and only if $\phi$ and $\psi$ are approximately unitarily equivalent relative to $\mathcal{K}(\mathcal{H})$.

In the following, we recall the definitions of inductive limit ${\rm C}^{\ast}$-algebras, AH algebras and certain useful properties.

\begin{remark}\label{AH-alg}
	By Proposition $6.2.4$ of \cite{Ror}, every inductive sequence of ${\rm C}^{\ast}$-algebras
	\begin{equation}
		\xymatrix{
		\mathcal{A}_{1} \ar[r]^{\phi_{_{1}}} & \mathcal{A}_{2} \ar[r]^{\phi_{_{2}}}& \mathcal{A}_{3} \ar[r]^{\phi_{_{3}}} &\cdots\\
		}
	\end{equation}
	has an inductive limit $(\mathcal{A},\{\varphi_{n}\}_{n\geq 1})$ which is also a ${\rm C}^{\ast}$-algebra such that
	\begin{enumerate}
		\item the diagram
		\begin{equation}\label{diagram-induc-lim}
		\xymatrix{
		\mathcal{A}_{n} \ar[r]^{\phi_{_{n}}} \ar[d]_{\varphi_{_{n}}} & \mathcal{A}_{n+1} \ar[dl]^{\varphi_{_{n+1}}}\\
		\mathcal{A}
		}
		\end{equation}
		commutes for each $n$ in $\mathbb{N}$, where $\phi_{_n}$'s and $\varphi_{_n}$'s are $\ast$-homomorphisms;
		\item  the ${\rm C}^{\ast}$-algebra $\mathcal{A}$ equals the norm-closure of the union of $\varphi_{n}(\mathcal{A}_{n})$ i.e.,
		\begin{equation}
			\mathcal{A}=\overline{\cup_{n\geq 1}\varphi_{n}(\mathcal{A}_{n})}^{\Vert\cdot\Vert}.
		\end{equation}
	\end{enumerate}
	Note that the diagram in $(\ref{diagram-induc-lim})$ implies that $\{\varphi_{n}(\mathcal{A}_{n})\}_{n\geq 1}$ forms a monotone increasing sequence of ${\rm C}^{\ast}$-algebras. 
	
	If this inductive limit ${\rm C}^{\ast}$-algebra $\mathcal{A}$ is a subalgebra of $(\mathcal{M},\tau)$ and $\mathcal{K}(\mathcal{M},\tau)$ is as in $(\ref{compact-ideal-R-tau})$, then Lemma $3.4.1$ of \cite{Davidson} entails that:
	\begin{equation}\label{induc-lim-ideal}
		\mathcal{A}\cap\mathcal{K}(\mathcal{M},\tau)=\overline{\cup^{}_{n\geq 1}(\mathcal{A}_{n}\cap\mathcal{K}(\mathcal{M},\tau))}^{\Vert\cdot\Vert}.
	\end{equation}

	By Definition $V.2.1.9$ of \cite{Blackadar}, a {separable} ${\rm C}^{\ast}$-algebra $\mathcal {A}$ is AH if it is $\ast$-isomorphic to an inductive limit of locally homogeneous ${\rm C}^{\ast}$-algebras {\rm (\emph{in the sense of $IV.1.4.1$ of} \cite{Blackadar})}. In addition, the following are useful:
	\begin{enumerate}
		\item[(i)] A ${\rm C}^{\ast}$-algebra $\mathcal {A}$ is  $n$-subhomogeneous, if and only if $\mathcal {A}^{\ast\ast}$ is a direct sum of Type ${\rm I}_m$ von Neumann algebras for $m\leq n$. In particular, $\mathcal {A}$ is  $n$-homogeneous, if and only if $\mathcal {A}^{\ast\ast}$ is a Type ${\rm I}_n$ von Neumann algebra {\rm (\emph{$IV.1.4.6$ of} \cite{Blackadar})};
		\item[(ii)] If $\varphi:\mathcal {A}\rightarrow\mathcal {B}$ is a bounded linear mapping between ${\rm C}^{\ast}$-algebra, then by general considerations $\varphi^{\ast\ast}:\mathcal {A}^{\ast\ast}\rightarrow\mathcal {B}^{\ast\ast}$ is a normal linear mapping of the same norm as $\varphi$. In addition, $\varphi^{\ast\ast}$ is a $\ast$-homomorphism if and only if $\varphi$ is a $\ast$-homomorphism. {\rm (\emph{$III.5.2.10$ of} \cite{Blackadar})}.
	\end{enumerate}
	
	As a quick application, if $\phi$ is a unital $\ast$-homomorphism of a unital locally homogeneous ${\rm C}^{\ast}$-algebra $\mathcal {A}$ into another unital ${\rm C}^{\ast}$-algebra $\mathcal {B}$, then $\phi(\mathcal{A})$ is also locally homogeneous.

For more about inductive limit, see Chapter XIV of \cite{Takesaki3} and Chapter $6$ of \cite{Ror}. It is convenient to assume that $\varphi_{n}$'s are injective $\ast$-homomorphisms and $\{\mathcal{A}_{n}\}_{n\geq 1}$ is an increasing sequence of ${\rm C}^{\ast}$-subalgebras of $\mathcal {A}$ whose union is norm-dense in $\mathcal {A}$.
\end{remark}

\section{Representations of AH algebras to type ${\rm II}_1$ factors}

In this section, we always assume that $(\mathcal{N},\tau)$ is a type ${\rm II}_1$ factor \emph{with separable predual}, where $\tau$ is the faithful, normal, tracial state. For two $*$-homomorphisms $\rho$ and $\pi$ of a unital ${\rm C}^{\ast}$-algebra $\mathcal{A}$ into $\mathcal{N}$, if there is a unitary operator $u$ in $\mathcal{N}$ such that the equality $u^*\rho(a)u=\pi(a)$ holds for every $a$ in $\mathcal{A}$, then $\rho$ and $\pi$ are {\emph{unitarily equivalent}} (denoted by $\rho\simeq\pi$ in $\mathcal{N}$). Let $\mathcal{A}_{+}$ denote the set of positive elements of $\mathcal{A}$.

The following Lemma \ref{Tool-3.1} and Lemma \ref{Takesaki-lemma} are prepared for Lemma \ref{Tool-3.3}.

\begin{lemma}\label{Tool-3.1}
	Let $C(X)$ be a unital, separable, abelian ${\rm C}^*$-algebra with $X$ a compact metric space. Suppose that $p$ is a projection in a type ${\rm II}_1$ factor $(\mathcal{N},\tau)$. 
	
	If $\pi: C(X) \rightarrow \mathcal{N}$ is a unital $\ast$-homomorphism and $\rho: C(X) \rightarrow p\mathcal{N}p$ is a unital $\ast$-homomorphism such
	that:	
	\begin{equation}\label{ineq-trace-rho-pi}
		\tau\Big(  R\left(\rho \left(  f\right)\right)\Big)\leq \tau\Big(R\left(\pi\left(f\right)\right)\Big), \quad\forall \ f \in C(X),
	\end{equation}
	then, for every positive function $h$ in $C(X)$,
	\begin{equation}\label{ineq-trace-positive}
		\tau\big(\rho \left(  h\right)\big)\leq \tau\big(\pi\left(h\right)\big).
	\end{equation}
\end{lemma}

\begin{proof}
	By applying Theorem II.2.5 of \cite{Davidson}, there are regular Borel measures $\mu_{\rho}$ and $\mu_{\pi}$ on $X$, such that $\rho$ (resp. $\pi$) extends to a ${\rm weak}^{*}$-{\scriptsize WOT} continuous  $\ast$-isomorphism $\tilde{\rho}$ (resp. $\tilde{\pi}$) of $L^{\infty}(\mu_{\rho})$ (resp. $L^{\infty}(\mu_{\pi})$) onto $\rho(C(X))^{\prime\prime}$ (resp. $\pi(C(X))^{\prime\prime}$).
	
	Let $\Delta$ be a Borel subset of $X$ and $\chi_{\Delta}$ be the characteristic function on $\Delta$. Note that, for each regular Borel measure $\mu$ on $X$, every $\mu$-measurable set is a disjoint union of a Borel set and a set of $\mu$-measure $0$. Thus we only need to concentrate on $\chi_{\Delta}$ for every Borel subset $\Delta$ of $X$ instead of considering measurable subsets.
	
	If $\Delta$ is a non-empty open subset of $X$, then there exists a positive function $f$ in $C(X)_{+}$ such that $f(\lambda)\neq 0$ for $\lambda\in\Delta$ and $f(\lambda)=0$ for $\lambda\in X\backslash\Delta$. The ${\rm weak}^{*}$-{\scriptsize WOT} continuity of $\tilde{\rho}$ entails that
	\begin{equation*}
		R(\rho(f))=\text{\scriptsize WOT-}\lim\nolimits_{n\rightarrow\infty}{\rho(f)^{\frac{1}{n}}}=\text{\scriptsize WOT-}\lim\nolimits_{n\rightarrow\infty}{\tilde{\rho}(f^{\frac{1}{n}})}=\tilde{\rho}(\chi_{\Delta}). 
	\end{equation*}
	Thus, the hypothesis in $(\ref{ineq-trace-rho-pi})$ implies that the inequality
	\begin{equation}\label{open-Delta}
		\tau\big(\tilde{\rho}(\chi_{\Delta})\big)\leq\tau\big(\tilde{\pi}(\chi_{\Delta})\big)
	\end{equation}
	holds for each open subset $\Delta$ of $X$. 
	
	If $\Delta$ is a Borel subset of $X$, then all the open subsets $O_{\alpha}$'s of $X$ with $\Delta\subseteq O_{\alpha}$, form a net with respect to each regular Borel measure $\mu$, i.e. $\mu(\Delta)=\lim\nolimits_{\alpha}\mu_{}(O_{\alpha})$. Let $\chi_{\alpha}$ be the characteristic function on $O_{\alpha}$. It follows that $\chi_{\alpha}$ converges to $\chi_{\Delta}$ in the $\mbox{weak}^{\ast}$ topology. Thus, the inequality in $(\ref{open-Delta})$ holds for every Borel subset $\Delta$ of $X$. 
	
	Given $\epsilon>0$ and a positive function $h$ in $C(X)$, there exist positive numbers $\lambda_1,\ldots,\lambda_m$ and a Borel partition $\Delta_1,\ldots,\Delta_m$ of $X$ such that
	\begin{equation*}
		\Vert h-\sum^{m}_{k=1}\lambda_k\chi_{\Delta_m}\Vert \leq \epsilon.
	\end{equation*} 
	It follows the inequality in $(\ref{ineq-trace-positive})$. This completes the proof.
\end{proof}

A lemma from \cite{Takesaki} is prepared as follows.

\begin{lemma}[Lemma 2.2 of \cite{Takesaki}]\label{Takesaki-lemma}
	Let $\mathcal{A}$ be a ${\rm C}^*$-algebra and $\{\pi,\H\}$ be a representation of $\A$. Then there is a unique linear mapping $\tilde{\pi}$ of the second conjugate space $\A^{**}$ of $\A$ onto $\pi(\A)^{\prime\prime}$ such that
	\begin{enumerate}
		\item the diagram
		\begin{equation}\label{ddual-mapping-pi}
		\xymatrix{
		 \mathcal{A}^{\ast\ast}  \ar[dr]^{\tilde{\pi}}\\
		 \mathcal{A}   \ar[u]_i   \ar[r]_{\pi\quad} & \pi(\mathcal{A})^{\prime\prime}
		}
	\end{equation}
	is commutative, where $i$ is the canonical imbedding of $\mathcal{A}$ into $\mathcal{A}^{\ast\ast}$.
		\item the mapping $\tilde{\pi}$ is continuous with respect to the $\sigma(\A ^{\ast\ast},\A ^{\ast})$-topology and the weak operator topology of  $\pi(\A)^{\prime\prime}$.
	\end{enumerate}
\end{lemma}

By virtue of Lemma \ref{Tool-3.1} and Lemma \ref{Takesaki-lemma}, we are ready for the following lemma.

\begin{lemma}\label{Tool-3.3}
	Let $\mathcal{A}$ be a separable ${\rm C}^*$-subalgebra in a type ${\rm II}_1$ factor $(\mathcal{N},\tau)$. Let $p$ be a projection in $(\mathcal{N},\tau)$. Suppose that $\pi:\mathcal{A} \rightarrow \mathcal{N}$ is a unital $\ast$-homomorphism and $\rho: \mathcal{A} \rightarrow p\mathcal{N}p$ is a unital $\ast$-homomorphism such
	that:	
	\begin{equation*}
		\tau\Big(  R\left(\rho \left(  a\right)\right)\Big)\leq \tau\Big(R\left(\pi\left( a \right)\right)\Big), \quad\forall \ a \in \mathcal{A}. 
	\end{equation*}
	Let $\tilde{\pi}:\mathcal{A}^{\ast\ast} \rightarrow \pi(\mathcal{A})^{\prime\prime}$ and $\tilde{\rho}:\mathcal{A}^{\ast\ast} \rightarrow \rho(\mathcal{A})^{\prime\prime}$ be the ${\rm weak}^{*}$-\text{\rm{\scriptsize WOT}} continuous $\ast$-homomorphisms extended by $\pi$ and $\rho$, respectively.
	Then, for every projection $e$ in $\mathcal{A}^{\ast\ast}$,
	\begin{equation*}\label{ineq-trace-rho-pi-ddual}
		\tau\left(\tilde{\rho} \left(  e \right)\right)\leq \tau\left(\tilde{\pi}\left( e \right)\right). 
	\end{equation*}
\end{lemma}

\begin{proof}
	As an application of Lemma \ref{Takesaki-lemma}, we have the following commutative diagram:
	\begin{equation}\label{ddual-mapping}
		\xymatrix{
		& \mathcal{A}^{\ast\ast} \ar[dl]_{\tilde{\rho}} \ar[dr]^{\tilde{\pi}}\\
		\rho(\mathcal{A})^{\prime\prime}& \mathcal{A} \ar[l]^{\quad  \rho}   \ar[u]_i   \ar[r]_{\pi\quad } & \pi(\mathcal{A})^{\prime\prime}
		}
	\end{equation}
	where $i$ is the canonical imbedding of $\mathcal{A}$ into $\mathcal{A}^{\ast\ast}$.
	
	Given a projection $e$ in $\mathcal{A}^{\ast\ast}$, by virtue of \cite[Theorem 1.6.5]{Kadison1} and Kaplansky's Density Theorem (Corollary 5.3.6 of \cite{Kadison1}), there exists a sequence $\{a_{n}\}_{n\geq 1}$ of positive operators in the unit ball of $\mathcal{A}$ such that $i(a_{n})$ is {\scriptsize SOT}-convergent to $e$. Then, Lemma 7.1.14 of \cite{Kadison2} entails that $\rho(a_{n})=\tilde{\rho}\circ i(a_{n})$ is {\scriptsize SOT}-convergent to $\tilde{\rho}(e)$ in $\rho(\mathcal{A})^{\prime\prime}$. Likewise, $\pi(a_{n})=\tilde{\pi}\circ i(a_{n})$ is {\scriptsize SOT}-convergent to $\tilde{\pi}(e)$ in $\pi(\mathcal{A})^{\prime\prime}$.
	
	In terms of Lemma \ref{Tool-3.1}, we have that
	\begin{equation*}
		\tau\left({\rho} \left(  a_{n} \right)\right)\leq \tau\left({\pi}\left( a_{n} \right)\right), \quad \forall\, n\geq 1. 
	\end{equation*}
	Since $\tau$ is a normal mapping, it follows that the inequality
	\begin{equation*}
		\tau\left(\tilde{\rho} \left(  e \right)\right)\leq \tau\left(\tilde{\pi}\left( e \right)\right) 
	\end{equation*}
	holds for every projection $e$ in $\mathcal{A}^{\ast\ast}$. This completes the proof.
\end{proof}

Lemma \ref{Tool-3.3} and the following Lemma \ref{Tool-3.4} are prepared for Lemma \ref{cor-3.8}. Note that the following Lemma \ref{Tool-3.4} is a routine calculation. For completeness, we sketch its proof.

\begin{lemma}\label{Tool-3.4}
	Let $(\mathcal{N},\tau)$ be a type ${\rm II}_1$ factor with tracial state $\tau$. Suppose that $\mathcal{A}$ is a unital separable ${\rm C}^{\ast}$-subalgebra of $(\mathcal{N}, \tau)$, $\ast$-isomorphic to $M_n(\mathbb{C})$, with an identity $I_{\mathcal{A}}$. Let $\phi$ and $\psi$ be $\ast$-homomorphisms of $\mathcal{A}$ into $\mathcal{N}$ such that
	\begin{equation}\label{trace-preserving-II1}
		\tau(\phi(a))=\tau(\psi(a)),\quad\forall \ a\in\mathcal{A}.
	\end{equation}
	Then, there exists a partial isometry $v$ in $\mathcal{N}$ such that
	\begin{equation}\label{unitary-implement-FD-alg-II1}
		\phi(a)=v^*\psi(a)v,\quad\forall a\in\mathcal{A}.
	\end{equation}
\end{lemma}

\begin{proof}
	Let $\{e_{ij}\}_{1\leq i,j\leq n}$ be a system of matrix units for $\mathcal{A}$ satisfying
	\begin{enumerate}
		\item $e^*_{ij}=e_{ji}$ for each  $i,j\in\NNN$;
		\item $e_{ij}e_{kl}=\delta_{jk}e_{il}$ for each  $i, j, k, l\in\NNN$;
		\item $\sum^{n}_{i=1}e_{ii}=I_{\A}$.
	\end{enumerate}
	Then $\{\phi(e_{ij})\}_{1\leq i,j\leq n}$ (resp. $\{\psi(e_{ij})\}_{1\leq i,j\leq n}$) is a system of matrix units for $\phi(\mathcal{A})$ (resp. $\psi(\mathcal{A})$). Note that if a certain $e_{i_0j_0}=0$, then each $e_{ij}=0$ for $1\leq i,j\leq n$. Thus, since $\tau(\phi(a))=\tau(\psi(a))$ for each $a\in\mathcal{A}$, we obtain $\ker\phi=\ker\psi$. This is because each element $a$ of $\mathcal{A}$ can be expressed as a matrix in terms of $\{e_{ij}\}_{1\leq i,j\leq n}$.
	
	Note that
	\begin{equation*}
		\tau(\phi(e_{ij}))=\tau(\psi(e_{ij}))<\infty, \quad\forall \ 1\leq i,j\leq n. 
	\end{equation*}
	Since $\mathcal{N}$ is a factor, there exists a partial isometry $v_{i}$ in $\mathcal{N}$ such that
	\begin{equation*}
		\phi(e_{ii})=v^*_{i}v_{i}\quad\text{ and }\quad \psi(e_{ii})=v_{i}v^*_{i}. 
	\end{equation*}
	Let $v$ be defined as
	\begin{equation*}
		v:=\sum\nolimits_{1\leq i\leq n}\phi(e_{i1})v^*_{1}\psi(e_{1i}). 
	\end{equation*}
	Then, it is routine to verify that
	\begin{enumerate}
		\item\quad $v^*v=I_{\psi(\mathcal{A})}$\quad and\quad $vv^*=I_{\phi(\mathcal{A})}$;
		\item\quad $\phi(e_{ij})v=v\psi(e_{ij})$ for $1\leq i,j\leq n$.
	\end{enumerate}
	This completes the proof.
\end{proof}

Recall that a locally homogeneous ${\rm C}^*$-algebra is a (finite) direct sum of homogeneous ${\rm C}^*$-algebra. A ${\rm C}^*$-algebra is homogeneous if it is $n$-homogeneous for some $n$. A ${\rm C}^*$-algebra $\A$ is $n$-homogeneous if every irreducible representation of $\A$ is of dimension $n$. The reader is referred to \cite[IV.1.4.1]{Blackadar} for the definition.

\begin{remark}\label{type-I_n-vN-alg}
	If follows from \cite[IV.1.4.6]{Blackadar} that if a ${\rm C}^*$-algebra is $n$-homogeneous, then its double dual is a type ${\rm I}_n$ von Neumann algebra. In the following lemmas, we will frequently mention type ${\rm I}_n$ von Neumann algebras. Thus the following facts about type ${\rm I}_n$ von Neumann algebras are useful. 
	
	Let $\mathcal{A}$ be a type ${\rm I}_{n}$ von Neumann algebra on a Hilbert space $\mathcal{H}$.  Then there exists a system of matrix units $\{e_{ij}\}_{1\leq i,j\leq n}$  for $\mathcal{A}$. Let $\R_{n}$ be the von Neumann algebra generated by $\{e_{ij}\}_{1\leq i,j\leq n}$. Then $\R_n$ is $\ast$-isomorphic to $\M_n(\CCC)$. Define $\P=\R_n^{\prime}\cap \A$. Since $\A$ is a type ${\rm I}_{n}$ von Neumann algebra, it follows that $\P$ is abelian and $\A=(\P\cup\R_n)^{\prime\prime}$. Moreover, $\A$ is $\ast$-isomorphic to the von Neumann tensor product $\P\otimes \M_n(\CCC)$. 
	
	The following observation is useful in the sequel. For each element $a$ in $\A$ and $\epsilon>0$, there are projections $p_1,\ldots,p_m$ in $\P$ with $1_{\P}=\sum_{1\leq i\leq m}p_i$ and matrices $a_1,\ldots,a_m$ in $\R_{n}$ such that
		\begin{equation}\label{appro-matrices}
			\Vert a-\sum_{1\leq i\leq m}p_i a_i\Vert<\epsilon.
		\end{equation}
\end{remark}

For Theorem \ref{thm-3.9}, we prepare the following lemma.

\begin{lemma}\label{cor-3.8}
	Let $\mathcal{A}$ be a unital, separable, locally homogeneous ${\rm C}^*$-subalgebra in a type ${\rm II}_1$ factor $(\mathcal{N},\tau)$. Let $p$ be a projection in $(\mathcal{N},\tau)$. Suppose that $\pi:\mathcal{A} \rightarrow \mathcal{N}$ is a unital $\ast$-homomorphism and $\rho: \mathcal{A} \rightarrow p\mathcal{N}p$ is a unital $\ast$-homomorphism such
	that:	
	\begin{equation*}
		\tau(  R\left(\rho \left(  a\right)\right))\leq \tau(R\left(\pi\left( a \right)\right)), \quad\forall \ a \in \mathcal{A}. 
	\end{equation*}
	Let $\tilde{\pi}$ (resp. $\tilde{\rho}$) be the ${\rm weak}^{*}$-{\rm{\scriptsize WOT}} continuous $\ast$-homomorphism of $\mathcal{A}^{\ast\ast}$ onto $\pi(\mathcal{A})^{\prime\prime}$ (resp. $\rho(\mathcal{A})^{\prime\prime}$) extended by the $\ast$-homomorphism $\pi$ (resp. $\rho$).
	
	For a finite subset $\mathcal{F}$ of $\mathcal{A}^{\ast\ast}$ and $\epsilon>0$, there exists a finite dimensional von Neumann subalgebra $\mathcal{B}$ in $\mathcal{A}^{\ast\ast}$ such that
	\begin{enumerate}
		\item  for each $a$ in $\mathcal{F}$, there is an element $b$ in $\mathcal{B}$ satisfying
	\begin{equation}
		\Vert a-b\Vert<\epsilon;
	\end{equation}
		\item  there is a unital $\ast$-homomorphism $\gamma:\mathcal{B}\rightarrow p^{\perp}\mathcal{N}p^{\perp}$ satisfying
		\begin{equation}
			\tilde{\rho}|_{\mathcal{B}}\oplus\gamma\simeq\tilde{\pi}|_{\mathcal{B}} \quad \mbox{ in }\ \mathcal{N}.
		\end{equation}
	\end{enumerate}
	Moreover, if $\gamma^{\prime}:\mathcal{B}\rightarrow p^{\perp}\mathcal{N}p^{\perp}$ is another unital $\ast$-homomorphism satisfying
		\begin{equation}
			\tilde{\rho}|_{\mathcal{B}}\oplus\gamma^{\prime} \simeq\tilde{\pi}|_{\mathcal{B}} \quad \mbox{ in }\ \mathcal{N},
		\end{equation}
		then $\gamma^{\prime}\simeq\gamma$ in $p^{\perp}\mathcal{N}p^{\perp}$.
\end{lemma}

\begin{proof}
	We first assume that $\mathcal{A}$ is $n$-homogeneous. It follows from Remark \ref{type-I_n-vN-alg} that the double dual $\A^{\ast\ast}$ of $\A$ can be expressed as $\mathcal{A}^{\ast\ast}=(\P\cup\R_n)^{\prime\prime}$ on some Hilbert space $\H$, where $\R_n$ is $\ast$-isomorphic to $\M_n(\CCC)$ and $\P=\R_n^{\prime}\cap \mathcal{A}^{\ast\ast}$ is an abelian von Neumann subalgebra.  Let $\mathcal{F}=\{a_1,\ldots,a_{k}\}$ be a finite subset of $\mathcal{A}$. Since $\A$ is isometrically imbedded into $\A^{\ast\ast}$, we can also view $a_i$ as an element in $\A^{\ast\ast}$ for each $1\leq i \leq k$. 
	
	For each $\epsilon>0$, by (\ref{appro-matrices}) in Remark \ref{type-I_n-vN-alg}, there are finitely many projections $p_{1},\ldots,p_{m}$ in $\P$ with $I_{\P}=\sum_{1\leq j\leq m}p_{j}$ and there are matrices $a_{i1},\ldots,a_{im}$ in $\R_{n}$ for $1\leq i\leq k$ such that
		\begin{equation}\label{appro-matrices-G}
			\Vert a_{i}-\sum_{1\leq j\leq m}p_{j} a_{ij}\Vert<\epsilon.
		\end{equation}
		
	Note that each $p_j$ is in the center of $\mathcal{A}^{\ast\ast}$ and $I_{\P}$ is the identity of $\mathcal{A}^{\ast\ast}$. Define a finite dimensional von Neumann subalgebra $\mathcal{B}$ in $\A^{\ast\ast}$ as follows
	\begin{equation}\label{B-decomp}
		\mathcal{B}:=\sum^{m}_{j=1}p_j\R_{n}.
	\end{equation}
	Thus, by virtue of Lemma \ref{Tool-3.3}, there is a finite dimensional von Neumann subalgebra $\mathcal{M}$ of $p^{\perp}\mathcal{N}p^{\perp}$ in the form
		\begin{equation}\label{M-decomp}
			\mathcal{M}:=\sum^{m}_{j=1}\mathcal{M}_{j}
		\end{equation}
		such that
	\begin{enumerate}
		\item for each $1\leq j\leq m$, $\mathcal{M}_{j}$ is $\ast$-isomorphic to $ M_{n}(\mathbb{C})$;
		\item the identity $q_j$ of $\M_j$ satisfies
		\begin{equation*}
			p^{\perp}=\sum_{1\leq j\leq m}q_j \quad \mbox{ and } \quad \tau(q_{j})=\tau\big(\tilde{\pi}(p_j)\big)-\tau\big(\tilde{\rho}(p_j)\big),\quad \mbox{ for }1\leq j\leq m.
		\end{equation*}
	\end{enumerate}
	Since each $\mathcal{M}_{j}$ is $\ast$-isomorphic to $ M_{n}(\mathbb{C})$, we obtain that $\mathcal{B}$ is $\ast$-isomorphic to $\mathcal{M}$. In terms of Lemma \ref{Tool-3.4}, we can define  a unital $\ast$-isomorphism $\gamma$ of $\B$ into $\M$ satisfying $\gamma(p_j)=q_j$ for each $1\leq j \leq m$, and
	\begin{equation*}
		\tilde{\rho}|_{\mathcal{B}}\oplus\gamma\simeq\tilde{\pi}|_{\mathcal{B}} \quad \mbox{ in }\ \mathcal{N}.
	\end{equation*}
	Moreover, $\gamma$ is unique  up to unitary equivalence.
	
	If $\A$ is a unital, separable, locally homogeneous ${\rm C}^{\ast}$-subalgebra of $\N$, then $\A$ can be expressed as $\A=\sum^{m}_{k=1} \A_k$ such that
	\begin{enumerate}
		\item for each $1\leq k \leq m$, $\A_k$ is $n_k$-homogeneous for some $n_k\in\NNN$,
		\item the identities $I_{\A_k}$ of $\A_k$ are mutually orthogonal.
	\end{enumerate}
	By composing the preceding arguments for each $\A_k$, we can complete the proof. 
\end{proof}

We are ready for the main theorem of this section. For a unital, separable ${\rm C}^{\ast}$-subalgebra $\A$ of $\N$ and two unital $\ast$-homomorphisms $\phi$ and $\psi$ of $\A$ into $\N$, recall that $\phi\sim_a \psi$ in $\N$ means the approximately unitary equivalence of $\phi$ and $\psi$ in $\N$, i.e., there exists a sequence of unitary operators $\{u_{n}\}^{\infty}_{n=1}$ in $\N$ such that
\begin{equation*}
	\lim_{n\to\infty}\Vert u^*_{n}\phi(a)u_{n}-\pi(a)\Vert=0,\quad \forall \ a\in \A.
\end{equation*}

\begin{theorem}\label{thm-3.9}
	Let $\mathcal{A}$ be a unital, separable AH subalgebra in a type ${\rm II}_1$ factor $(\mathcal{N},\tau)$. Let $p$ be a projection in $(\mathcal{N},\tau)$. Suppose that $\pi:\mathcal{A} \rightarrow \mathcal{N}$ is a unital $\ast$-homomorphism and $\rho: \mathcal{A} \rightarrow p\mathcal{N}p$ is a unital $\ast$-homomorphism such
	that:	
	\begin{equation*}
		\tau(  R\left(\rho \left(  a\right)\right))\leq \tau(R\left(\pi\left( a \right)\right)), \quad\forall a \in \mathcal{A}. 
	\end{equation*}
	Then, there exists a unital $\ast$-homomorphism $\gamma: \mathcal{A} \rightarrow p^{\perp}\mathcal{N}p^{\perp}$ such that
	\begin{equation}
		\rho\oplus\gamma\sim_{a}\pi \quad\mbox{ in }\ \mathcal{N}.
	\end{equation}
\end{theorem}

\begin{proof}
	As in Remark \ref{AH-alg}, we can assume that
	\begin{equation}
		\mathcal{A}=\overline{\cup_{n\geq 1}\mathcal{A}_{n}}^{\Vert\cdot\Vert},
	\end{equation}
	where $\{\mathcal{A}_{n}\}_{n\geq 1}$ is a monotone increasing sequence of unital, locally homogeneous ${\rm C}^*$-algebras.
	
	Since $\mathcal{A}$ is separable, let $\mathcal{F}_1\subseteq\mathcal{F}_2\subseteq\cdots$ be a monotone increasing sequence of finite subsets in $\cup_{n\geq 1}\mathcal{A}_{n}$ such that $\cup_{i\geq 1}\mathcal{F}_{i}$ is $\Vert\cdot\Vert$-dense in $\mathcal{A}$, where $\Vert\cdot\Vert$ is the operator norm. By dropping to a subsequence, we can assume that $\mathcal{F}_i\subset \mathcal{A}_{i}$ for each $i$ in $\mathbb{N}$. We further require that $1_{\mathcal{A}}\in\mathcal{F}_1$.

	In terms of \cite[III.5.2.10]{Blackadar}, for each $n\in\NNN$, we have that
	\begin{equation*}
		\A_n\subseteq\A_{n+1}\subseteq\A \quad  \implies \quad \A^{\ast\ast}_n\subseteq\A^{\ast\ast}_{n+1}\subseteq\A^{\ast\ast}.
	\end{equation*} 
	In the following, we identify $\A$ as a unital, separable ${\rm C}^*$-subalgebra of $\A^{\ast\ast}$. Note that each $\mathcal{A}^{\ast\ast}_{i}$ is a direct sum of finitely many type ${\rm I}_m$ von Neumann algeba for $m$ less than a certain $n$.
	
	As an application of Lemma \ref{cor-3.8}, there is a finite dimensional von Neumann subalgebra $\mathcal{B}_{1}$ of $\mathcal{A}_{1}^{\ast\ast}$ corresponding to $\mathcal{F}_{1}$ such that for each $a$ in $\mathcal{F}_{1}$, there is an $a_{1}$ in $\mathcal{B}_{1}$ satisfying
	\begin{equation*}\label{ineq-F1-B1}
		\Vert a-a_1\Vert<\frac{1}{2^{2}}.
	\end{equation*}
	Applying Lemma \ref{cor-3.8} once more for $\F_2$ and the system of matrix units of $\B_1$, there exists a finite dimensional von Neumann subalgebra $\mathcal{B}_{2}$ of $\mathcal{A}_{2}^{\ast\ast}$ such that
		\begin{enumerate}
		\item  $\mathcal{B}_{2}\supseteq\mathcal{B}_{1}$;
		\item  for each $a$ in $\mathcal{F}_{2}$, there is an $a_{2}$ in $\mathcal{B}_{2}$ satisfying
	\begin{equation*}\label{ineq-F2-B2}
		\Vert a-a_2\Vert<\frac{1}{2^{3}}.
	\end{equation*}
	\end{enumerate}
	By induction, there is a finite dimensional von Neumann subalgebra $\mathcal{B}_{i}$ of $\mathcal{A}_{i}^{\ast\ast}$ corresponding to each $\mathcal{F}_{i}$ such that
	\begin{enumerate}
		\item  $\mathcal{B}_{i}\supseteq\mathcal{B}_{i-1}$ for each $i\geq 2$;
		\item  for each $a$ in $\mathcal{F}_{i}$, there is an $a_{i}$ in $\mathcal{B}_{i}$ satisfying
	\begin{equation}\label{ineq-Fi-Bi}
		\Vert a-a_i\Vert<\frac{1}{2^{i+1}}.
	\end{equation}
	\end{enumerate}
	
	Moreover, there is a unital $\ast$-homomorphism $\phi_i:\mathcal{B}_i\rightarrow p^{\perp}\mathcal{N}p^{\perp}$ such that
	\begin{equation}\label{phi_i-unitary-equi}
		\tilde{\rho}|_{\mathcal{B}_i}\oplus\phi_i\simeq\tilde{\pi}|_{\mathcal{B}_i} \quad \mbox{ in }\ \mathcal{N}.
	\end{equation}
	Note that $\mathcal{B}_{i+1}\supseteq\mathcal{B}_{i}$ implies that
	\begin{equation*}
		\tilde{\rho}|_{\mathcal{B}_i}\oplus\phi_{i+1}|_{\mathcal{B}_i}\simeq\tilde{\rho}|_{\mathcal{B}_i}\oplus\phi_i\simeq\tilde{\pi}|_{\mathcal{B}_i} \quad \mbox{ in }\ \mathcal{N}. 
	\end{equation*}
	Thus we have $\phi_{i+1}|_{\mathcal{B}_i}\simeq\phi_i\quad \mbox{ in }\ p^{\perp}\mathcal{N}p^{\perp}$.
	
	Define $\gamma_{1}:=\phi_{1}$. Let $u_{2}$ be the unitary operator in $p^{\perp}\mathcal{N}p^{\perp}$ such that $u^{\ast}_{2}(\phi_{2}|_{\mathcal{B}_{1}})u_{2}=\gamma_{1}$
	and define
	\begin{equation*}
		\gamma_{2}(\cdot):=u^{\ast}_{2}\phi_{2}(\cdot)u_{2}. 
	\end{equation*}
	Likewise, let $u_{i+1}$ be the unitary operator in $p^{\perp}\mathcal{N}p^{\perp}$ such that $u^{\ast}_{i+1}(\phi_{i+1}|_{\mathcal{B}_{i}})u_{i+1}=\gamma_{i}$
	and define
	\begin{equation}\label{phi_i-gamma_i}
		\gamma_{i+1}(\cdot):=u^{\ast}_{i+1}\phi_{i+1}(\cdot)u_{i+1}.
	\end{equation}
	With respect to the choice of the family $\{\B_{i}\}^{\infty}_{i=1}$ of finite dimensional von Neumann algebras, we construct a sequence of $\ast$-homomorphisms $\{\gamma_{i}\}_{i\geq 1}$  such that the equality
	\begin{equation}\label{gamma_i+k-gamma_i}
		\gamma_{i+k}(b)=\gamma_{i}(b)
	\end{equation}
	holds for every $b$ in $\mathcal{B}_{i}$ and each $i \geq 1$, $k\geq 1$.
	
	By virtue of (\ref{ineq-Fi-Bi}), for every $a$ in $\mathcal{F}_{i}$, there is a sequence $\{a_{k}:a_{k}\in\mathcal{B}_{k}\}_{k\geq 1}$ such that
	\begin{equation}\label{A_k-Cauchy}
			a_k=0\quad \mbox{ for } k<i, \quad\mbox{ and }\quad \Vert a-a_k\Vert<\frac{1}{2^{k+1}}\quad \mbox{ for } k\geq i.
	\end{equation}
	It follows that $\{a_k\}^{\infty}_{k= 1}$ is a Cauchy sequence  in the operator norm topology. Moreover, we assert that $\{\gamma_k(a_k)\}^{\infty}_{k= 1}$ is a Cauchy sequence in the operator norm topology. Notice that, for $k\geq i$ and each $p\geq 1$, (\ref{gamma_i+k-gamma_i}) and (\ref{A_k-Cauchy}) imply that
	\begin{equation*}
		\Vert \gamma_{k+p}(a_{k+p})-\gamma_{k}(a_k)\Vert= \Vert \gamma_{k+p}(a_{k+p})-\gamma_{k+p}(a_k)\Vert\leq\Vert a_{k+p}-a_k\Vert<\frac{1}{2^k}.    
	\end{equation*}
	This guarantees that $\{\gamma_{k}(a_k)\}_{k\geq 1}$ is also a Cauchy sequence in the operator norm topology. 
	
	Note that, for each fixed $a$ in $\mathcal{F}_{i}$, if there is another sequence $\{a^{\prime}_{k}:a^{\prime}_{k}\in\mathcal{B}_{k}\}$ satisfying
	\begin{equation*}\label{A'_k-Cauchy}
			a^{\prime}_k=0\quad \mbox{ for } k<i, \quad\mbox{ and }\quad \Vert a-a^{\prime}_k\Vert<\frac{1}{2^{k+1}}\quad \mbox{ for } k\geq i, 
	\end{equation*}
	then both $\{a^{\prime}_k\}^{\infty}_{k= 1}$ and $\{\gamma_{k}(a^{\prime}_k)\}^{\infty}_{k= 1}$ are Cauchy sequences in the operator norm topology. Furthermore, the limit $\lim_{k\to\infty}\Vert a_k-a^{\prime}_k\Vert=0$ entails that the equality
	\begin{equation*}
		\lim\nolimits_{k\rightarrow\infty}\gamma_{k}(a^{\prime}_{k}) =\lim\nolimits_{k\rightarrow\infty}\gamma_{k}(a_{k}) 
	\end{equation*}
	holds in the operator norm topology. Since $\Vert \gamma_{k}\Vert\leq 1$ for each $k\in\NNN$, the mapping
	\begin{equation*}
		\gamma:\cup_{i\geq 1}\mathcal{F}_{i}\rightarrow p^{\perp}\mathcal{N}p^{\perp},\qquad \gamma(a):=\lim\nolimits_{k\rightarrow\infty}\gamma_{k}(a_{k}) 
	\end{equation*}
	extends to a well-defined unital $\ast$-homomorphism of $\mathcal{A}$ into $p^{\perp}\mathcal{N}p^{\perp}$.
	
	Note that, for each $i\geq 1$, (\ref{phi_i-unitary-equi}) and (\ref{phi_i-gamma_i}) entail that there is a unitary operator $v_i$ in $\mathcal{N}$ such that $v^{\ast}_{i}(\rho_{i}\oplus\gamma_{i})v^{}_{i}=\pi_{i}$, where $\rho_{i}$ (resp. $\pi_{i}$) is the restriction of $\tilde{\rho}$ (resp. $\tilde{\pi}$) on $\mathcal{B}_{i}$.	Thus, for each $a\in\mathcal{F}_{i}$, it follows that
	\begin{equation*}
		\Vert \pi(a)-v^{\ast}_{i}(\rho(a)\oplus\gamma(a))v^{}_{i}\Vert\leq\Vert \pi(a)-\pi_{i}(a_{i})\Vert+ \Vert\rho_{i}(a_{i})\oplus\gamma_{i}(a_{i})-\rho(a)\oplus\gamma(a)\Vert<\frac{1}{2^{i}}. 
	\end{equation*}
	Since $\cup_{i\geq 1}\mathcal{F}_{i}$ is $\Vert\cdot\Vert$-dense in $\mathcal{A}$, we obtain that
	\begin{equation*}
 	\lim\nolimits_{i\rightarrow\infty}\Vert \pi(a)-v^{\ast}_{i}(\rho(a)\oplus\gamma(a))v^{}_{i}\Vert=0,\quad \forall \ a\in\mathcal{A}. 
 \end{equation*}
	This completes the proof.
\end{proof}

\section{Representations of AH algebras to semifinite, properly infinite factors}

Let $(\mathcal{M},\tau)$  be a countably decomposable von Neumann factor with a faithful, normal, semifinite, tracial weight $\tau$. Recall that $\F(\M,\tau)$ is the set of all $(\mathcal{M},\tau)$-finite-rank operators in $\M$ and $\K(\M,\tau)$ is the $\Vert\cdot\Vert$-norm closure of $\F(\M,\tau)$, where $\Vert\cdot\Vert$ denotes the operator norm. See Definition \ref{M-rank-opts} in Section 2 for details.

Let $\mathcal{A}$ be a separable AH subalgebra of $\mathcal{M}$ with an identity $I_{\mathcal{A}}$. Let $\phi$ and $\psi$ be unital $\ast$-homomorphisms of $\mathcal{A}$ into $\mathcal{M}$. The main goal of this section is to prove the equivalence of the following statements:
\begin{enumerate}
	\item [(1)] \quad $\phi \sim_{a}\psi$ \ in $\mathcal{M}$, i.e., $\phi$ and $\psi$ are approximately unitarily equivalent in $\M$;
	\item [(2)] \quad $\phi \sim_{\mathcal{A}}\psi \mod\ \mathcal{K}(\mathcal{M},\tau)$ (see Definition \ref{AEC}).
\end{enumerate}   

Note that the analogue in $\B(\H)$ is proved in three steps. First, a separable ${\rm C}^{\ast}$-algebra is cut into two parts in terms of all compact operators in the ${\rm C}^{\ast}$-algebra, i.e., the part containing all compact operators in the ${\rm C}^{\ast}$-algebra and the part containing no compact operators. Then, the equivalence of (1) and (2) is proved with respect to the two parts, respectively. 

The proof with respect to the `non-compact' part of a separable ${\rm C}^{\ast}$-algebra in $\B(\H)$ is due to Voiculescu \cite{Voi2}, sometimes called the absorption theorem. His proof works for every separable ${\rm C}^{\ast}$-algebra in $\B(\H)$. Recently, the authors in \cite{Li} extends the preceding absorption theorem for separable nuclear ${\rm C}^{\ast}$-algebras in $(\M,\tau)$.

 Notice that the proof associated with  the `compact' part of a separable ${\rm C}^{\ast}$-algebra in $\B(\H)$ is based on minimal projections. This leads to the first obstruction that $(\M,\tau)$ may contain no minimal projections. Meanwhile, for a separable AH algebra in $(\M,\tau)$, there might not be sufficiently many projections in the AH algebra, generally.  This is the second obstruction.  Thus, it is necessary to develop new techniques to overcome these two obstructions.
 In this section, Lemma \ref{Tool-1.4} is devoted to cut each ${\rm C}^{\ast}$-algebra $\A$ according to  $\A\cap\K(\mathcal{M},\tau)$. By virtue of a series of lemmas, the equivalence of (1) and (2) mentioned above with respect to the `compact' part is proved in Theorem \ref{Tool-1.6}. Combining with Theorem \ref{VoiThmNuclear} (\cite[Theorem 5.3.1]{Li}), the equivalence of (1) and (2)  is proved in Theorem \ref{AH-main-thm}.

\begin{lemma}\label{normal-extension}
	For $i=1,2$, let $\mathcal{M}_i$  be a   von Neumann algebra with a faithful, normal, semifinite, tracial weight $\tau_i$. Let $\mathcal{F}(\mathcal{M}_i, \tau_i)$ be the set of all  $(\M_i,\tau_i)$-finite-rank operators in $(\mathcal{M}_i,\tau_i)$. 
	
	Assume that $\mathcal{A}_i$  is a $\ast$-subalgebra of $\mathcal{F}(\mathcal{M}_i, \tau_i)$ such that $\mathcal{A}_i$  is ${\text {weak}}^\ast$-dense in $\mathcal{M}_i$, for $i=1,2$.
	
	If $\rho: \mathcal{A}_1\rightarrow \mathcal{A}_2$ is a $\ast$-isomorphism such that
	\begin{equation}\label{tau1-rho-tau2}
		\tau_2(\rho(x))=\tau_1(x),\quad\forall\ x\in\mathcal{A}_1,
	\end{equation}
	then $\rho$ extends uniquely to a normal $\ast$-isomorphism $\rho':\mathcal{M}_1\rightarrow \mathcal{M}_2$ satisfying that
	\begin{equation*}
		\tau_2(\rho'(x))=\tau_1(x),\quad\forall\ 0<x\in\mathcal{M}_1. 
	\end{equation*}
\end{lemma}

\begin{proof}
	Since $\tau_{i}$ is a faithful, normal, semifinite tracial weight on $\mathcal{M}_i$ for $i=1,2$, then
	\begin{equation}\label{inner-product}
		(a,b):=\tau_{i}(b^{\ast}a),
	\end{equation}
	defines a definite inner product on $\mathcal{A}_i$. Let $\mathcal{H}_i$ be the completion of $\mathcal{A}_i$ relative to the norm associated with the inner product defined in (\ref{inner-product}). Denote by $L^2(\mathcal M_i, \tau_i)$ the completion of $\{x:x\in \M_i,\ \tau_i(x^*x)<\infty\}$ relative to the norm associated with the inner product  in (\ref{inner-product}). From the fact that each $\mathcal{A}_i$  is ${\text {weak}}^\ast$-dense in $\mathcal{M}_i$, it follows that   $\mathcal{H}_i=L^2(\mathcal M_i, \tau_i).$  By applying Theorem 7.5.3 of \cite{Kadison2}, the faithful, normal tracial weight $\tau_{i}$ induces a faithful, normal, representation $\pi_i$ of $\mathcal{M}_i$ on $\mathcal{H}_i$ for $i=1,2$.
	
	Let $\{a_{\lambda}\}_{\lambda\in\Lambda}$ be a bounded net in $\mathcal{A}_1$ such that $a_{\lambda}$ converges to $a\in\mathcal{M}_1$ in the ${\rm weak}^{\ast}$ topology. Note that, for each $b$ in $\mathcal{A}_1$, the equality
	\begin{equation*}\label{pi1-pi2}
		(\pi_1(a_{\lambda})b,b)=\tau_1(b^*a_{\lambda}b)=\tau_2(\rho(b^*a_{\lambda}b))=(\pi_2(\rho(a_{\lambda}))\rho(b),\rho(b)) 
	\end{equation*}
	entails that $\pi_2(\rho(a_{\lambda}))$ converges to an operator $x$ in $\pi_2(\mathcal{M}_2)$ in the weak operator topology. Note that $\pi_2$ is a normal $\ast$-isomorphism between von Neumann algebras $\mathcal{M}_2$ and $\pi_2(\mathcal{M}_2)$. It follows that $\rho(a_{\lambda})$ converges to $\pi^{-1}_{2}(x)$  in the weak operator topology.  For $a$, the ${\rm weak}^{\ast}$ limit of $a_{\lambda}$, in $\mathcal{M}_1$, define $\rho'(a):=\pi^{-1}_2(x)$. It is easily verified that $\rho^{\prime}$ is well-defined. In this way, $\rho$ extends uniquely to a normal $\ast$-isomorphism $\rho':\mathcal{M}_1\rightarrow \mathcal{M}_2$. Combining with the fact that each $\tau_i$ is normal, we can further conclude that
	\begin{equation*}
		\tau_2(\rho'(x))=\tau_1(x),\quad\forall\ 0<x\in\mathcal{M}_1. 
	\end{equation*}
	This completes the proof.
\end{proof}

\begin{lemma}\label{Tool-rep-C^star-alg}
	Let $\mathcal{M}$ be a von Neumann algebra with a faithful, normal, semifinite, tracial weight $\tau$. Suppose that $\mathcal{A}$ is a $\ast$-subalgebra of $\mathcal{F}(\mathcal{M}, \tau)$ and $\mathcal A$ is a linear span of  $\mathcal A_+$.
	
	If $\rho$ is a $\ast$-isomorphism of $\mathcal{A}$ into $\mathcal{M}$ such that $\rho$ and the identity mapping $id$ of $\mathcal{A}$ are approximately equivalent in $\mathcal{M}$, written as
	\begin{equation*}
		id\sim_{a}\rho, \quad \mbox{in }\mathcal{M}, 
	\end{equation*}
	then $\rho$ extends uniquely to a normal $\ast$-isomorphism $\rho'$ of the {\scriptsize WOT}-closure of $\mathcal{A}$ into $\mathcal{M}$ such that
	\begin{equation*}
		\tau(\rho'(a))=\tau(a),\quad  \text{ for }  \text { each postive operator $a$ in the {\scriptsize WOT}-closure of $\mathcal{A}$}. 
	\end{equation*}
\end{lemma}

\begin{proof}
	We will apply Lemma \ref{normal-extension} to extend $\rho$ uniquely to a von Neumann algebra isomorphism of the {\scriptsize WOT}-closure of $\mathcal{A}$ to the {\scriptsize WOT}-closure of $\rho(\mathcal{A})$ with the desired property. It is sufficient to prove the following equality:
	 \begin{equation}\label{trace-equ}
		\tau(a)=\tau(\rho(a)),\quad\forall\ a\in\mathcal{A}_{+}.
	\end{equation}

	Recall that $\mathcal{PF}(\mathcal{M},\tau)$ is the set of projections $p$ in $\mathcal{M}$ with $\tau(p)<\infty$. Let $a>0$ be a $(\mathcal{M},\tau)$-finite-rank operator in $\mathcal{A}$. Note that
	\begin{equation*}
		\tau{(a)}=\sup\{\tau(ap): p\in\mathcal{PF}(\mathcal{M},\tau)\}. 
	\end{equation*}
	We claim that $\tau(\rho(a)p)\leq \tau(a)$ for each finite trace projection $p$ in $\mathcal{M}$.
	
	Since $\rho$ and the identity mapping $id$ are approximately equivalent in $\mathcal{M}$, there exists a sequence of unitary operators $\{u_k\}^{}_{k\geq 1}$ in $\mathcal{M}$ such that
	\begin{equation*}
		\lim_{k\rightarrow\infty}\Vert\rho(a)-u^*_kau_k\Vert=0.
	\end{equation*}
	Let $p$ be a finite trace projection in $\mathcal{M}$. In terms of the Holder inequality, we have
	\begin{equation*}
		\vert\tau\big((\rho(a)-u^*_kau_k)p\big)\vert\leq \Vert\rho(a)-u^*_kau_k\Vert\Vert p\Vert_1\rightarrow 0, \ \mbox{ as } k\rightarrow \infty.
	\end{equation*}
	This implies that
	\begin{equation*}
		\vert\tau(\rho(a)p)\vert=\lim_{k\rightarrow\infty}\vert\tau((u^*_k a u_k)p)\vert\leq \Vert u^*_kau_k\Vert_1\Vert p\Vert=\tau(a).
	\end{equation*}
	This completes the proof of the claim.
	
	Since
	\begin{equation*}
		\tau{(\rho(a))}=\sup\{\tau(\rho(a)p): p\in\mathcal{PF}(\mathcal{M},\tau)\}, 
	\end{equation*}
	it follows that
	\begin{equation*}
		\tau(\rho(a))\leq\tau(a). 
	\end{equation*}

	Similarly, we have
	\begin{equation*}
		\tau(a)\leq\tau(\rho(a)). 
	\end{equation*}
	Thus, we have
	\begin{equation*}
		\tau(a)=\tau(\rho(a)). 
	\end{equation*}
	This completes the proof of (\ref{trace-equ}). Thus, by virtue of Lemma \ref{normal-extension}, $\rho$ can be extended uniquely to a von Neumann algebra isomorphism $\rho^{\prime}$ of the {\scriptsize WOT}-closure of $\mathcal{A}$ to the {\scriptsize WOT}-closure of $\rho(\mathcal{A})$ with the desired property.
\end{proof}

\begin{remark}\label{AH}
	Let $\mathcal{M}$ be a von Neumann algebra with a faithful, normal, semifinite, tracial weight $\tau$. Let $\mathcal{A}$ be a $(separable)$ AH subalgebra of $(\mathcal{M},\tau)$. It is convenient to assume that there exists an increasing sequence $\{\mathcal{A}_{n}\}^{}_{n\geq 1}$ of locally homogeneous ${\rm C}^{\ast}$-algebras, as in Remark $\ref{AH-alg}$, such that
	\begin{equation}\label{Induc-lim-def}
		\mathcal{A}=\overline{\cup^{}_{n\geq 1}\mathcal{A}_{n}}^{\Vert\cdot\Vert}.
	\end{equation}
	By applying Lemma $3.4.1$ of \cite{Davidson}, we have
	\begin{equation}
		\mathcal{A}\cap\mathcal{K}(\mathcal{M},\tau)=\overline{\cup^{}_{n\geq 1}(\mathcal{A}_{n}\cap\mathcal{K}(\mathcal{M},\tau))}^{\Vert\cdot\Vert}.
	\end{equation}
\end{remark}
	
Let $\mathcal{M}$ be a von Neumann algebra with a faithful, normal, semifinite, tracial weight $\tau$. For a separable, unital $\rm{C^{\ast}}$-subalgebra $\mathcal{A}$ in $(\mathcal{M},\tau)$, the reader is referred to \cite[Lemma 3.1]{Hadwin3} and \cite[Lemma 3.2]{Hadwin3} for several useful properties for operators in $\K(\M,\tau)$. By a routine continuous function calculus and \cite[Lemma 3.1]{Hadwin3}, we construct a sequence of $(\mathcal{M},\tau)$-finite-rank operators $\Vert\cdot\Vert$-norm dense in $\mathcal{A}_{+}\cap\mathcal{K}(\mathcal{M},\tau)$. Define
	\begin{equation}\label{proj-finite-rank-union}
		p_{\mathcal{K}(\mathcal{A},\tau)}:=\bigvee^{}_{x\in\mathcal{A}\cap\mathcal{K}(\mathcal{M},\tau)}R(x).
	\end{equation}
	In the following lemma, we prove that the projection $p_{\mathcal{K}(\mathcal{A},\tau)}$ reduces $\mathcal{A}$.

\begin{lemma}\label{Tool-1.4}
Let $(\mathcal{M},\tau)$ be a von Neumann algebra with a faithful, normal, semifinite, tracial weight $\tau$. Suppose that $\mathcal{A}$ is a separable ${\rm C}^{\ast}$-subalgebra of $\mathcal{M}$. Let $\{x_{n}\}^{\infty}_{n= 1}$ be a sequence of positive, $(\mathcal{M},\tau)$-finite-rank operators in the unit ball of $\mathcal{A}_{+}\cap\mathcal{F}(\mathcal{M},\tau)$ such that  $\{x_{n}\}^{\infty}_{n= 1}$ is $\Vert\cdot \Vert$-dense in the unit ball of $\mathcal{A}_{+}\cap\mathcal{K}(\mathcal{M},\tau)$, where $\Vert\cdot \Vert$ is the operator norm. Then the following statements are true:
\begin{enumerate}
	\item $p_{\mathcal{K}(\mathcal{A},\tau)}=\vee^{}_{n\geq 1}R(x_n)$, where $p_{\mathcal{K}(\mathcal{A},\tau)}$ is defined as in $(\ref{proj-finite-rank-union})$;
	\item $p_{\mathcal{K}(\mathcal{A},\tau)}x=xp_{\mathcal{K}(\mathcal{A},\tau)},\ \ \forall\, x\in\mathcal{A}$.
\end{enumerate}
\end{lemma}

\begin{proof}
	Assume that $\mathcal{A}\cap\mathcal{K}(\mathcal{M},\tau)\neq \mathbf{0}$ and the von Neumann algebra $\mathcal{M}$ acts on a Hilbert space $\mathcal{H}$.
		
	Let $p$ be the union of the range projections $R(x_n)$ of all these positive, finite-rank operators $x_n$'s. For each element $x$ in the unit ball of $\mathcal{A}\cap\mathcal{K}(\mathcal{M},\tau)$, let $x=\vert x^{\ast}\vert v$ be the polar decomposition of $x$ (see  Theorem 6.1.2 of \cite{Kadison2}). For every $\epsilon>0$,  there is a positive, finite-rank operator $x_{n}$ such that $\Vert \vert x^{\ast}\vert-x_{n}\Vert\leq \epsilon$. Thus, for each unit vector $\xi$ in $\mathcal{H}$, it follows that
	\begin{equation*}
		\Vert x\xi-x_{n}v\xi\Vert\leq\Vert \vert x^{\ast}\vert-x_{n}\Vert\leq \epsilon. 
	\end{equation*}
	Since $px_{n}v\xi=x_{n}v\xi$, it follows, eventually, that $px=x$ holds for every $x$ in $\mathcal{A}\cap\mathcal{K}(\mathcal{M},\tau)$. Thus, we have that
	\begin{equation*}
		p=p_{\mathcal{K}(\mathcal{A},\tau)}.	
	\end{equation*}
		
	In the following, assume contrarily that there exists an operator $a$ in $\mathcal{A}$ such that $pap^{\perp}\neq 0$. Then there exists a positive, $(\M,\tau)$-finite-rank operator $a_1$ such that $R(a_1)ap^{\perp}\neq 0$. Since the restriction of each bounded linear positive operator on the closure of its range is injective, the equality 
	\begin{equation*}
		\mbox{ker}(a_1)=\mbox{ker}(a_1^{\frac{1}{2}})
	\end{equation*}
	entails that
	\begin{equation*}
		R(a_1^{\frac{1}{2}})ap^{\perp}=R(a_1)ap^{\perp}\neq 0 \quad\mbox{ and }\quad a_1^{\frac{1}{2}}ap^{\perp}=a_1^{\frac{1}{2}}R(a_1^{\frac{1}{2}})ap^{\perp}\neq 0. 
	\end{equation*}
	Thus $p^{\perp}a^{\ast}a_1ap^{\perp}\neq 0$ implies $pa^{\ast}a_1a\neq a^{\ast}a_1a$. Note that the inequality $\tau(R(a^{\ast}a_1a))<\infty$ ensures that $a^{\ast}a_1a$ is a positive, $(\M,\tau)$-finite-rank operator. Then the fact that $pa^{\ast}a_1a\neq a^{\ast}a_1a$ contradicts the definition of $p$. It follows that $p$ reduces $\mathcal{A}$. This completes the proof.
\end{proof}

Let $\mathcal{M}$ be a von Neumann algebra and $\A$ a ${\rm C}^{\ast}$-subalgebra of $\mathcal{M}$. Recall that by $W^*(\A)$ we denote the {\scriptsize WOT}-closure of $\mathcal{A}$, which is also the von Neumann algebra generated by $\mathcal{A}$. 

As mentioned in Remark \ref{AH}, a separable AH algebra is a inductive limit of locally homogeneous ${\rm C}^{\ast}$-algebras. The following three lemmas are developed with respect to locally homogeneous ${\rm C}^{\ast}$-algebras in $(\M,\tau)$, which are prepared for Theorem \ref{Tool-1.6}.

\begin{lemma}\label{Tool-1.5}
	Let $\mathcal{M}$ be a  von Neumann algebra with a faithful, normal, semifinite, tracial weight $\tau$. Suppose that $\mathcal{A}$ is a (separable) locally homogeneous ${\rm C}^{\ast}$-subalgebra of $\mathcal{M}$.
	
	Let $a$ be a non-zero positive finite $\M$-rank operator in $\mathcal{A}\cap\mathcal{F}(\mathcal{M},\tau)$. Then there exists a central projection $e$ of $W^{\ast}(\mathcal{A})$, the {\scriptsize WOT}-closure of $\mathcal{A}$ in $\mathcal{M}$, such that
	\begin{enumerate}
		\item \quad $e\in\mathcal{F}(\mathcal{M},\tau)$;
		\item \quad $R(a)\leq e$;
		\item \quad $W^{\ast}(\mathcal{A})e\subseteq W^{\ast}(\mathcal{A}\cap\mathcal{F}(\mathcal{M},\tau))$.
	\end{enumerate}
\end{lemma}

\begin{proof}
	By Definition IV.1.4.1 of \cite{Blackadar}, a ${\rm C}^{\ast}$-algebra $\mathcal{A}$ is locally homogeneous, if it is a finite direct sum of homogeneous ${\rm C}^{\ast}$-algebras. For the sake of simplicity, we assume $\mathcal{A}$ to be $n$-homogeneous.
	
	\vspace{0.2cm}

	\textit{Claim $\ref{Tool-1.5}.1.$ } \emph{If $\mathcal{A}$ is $n$-homogeneous, then $W^{\ast}(\mathcal{A})$ is a type ${\rm I}_{n}$ von Neumann algebra.}
	
	\vspace{0.2cm}
	
	Let $id:\mathcal{A}\rightarrow\mathcal{A}$ be the identity mapping of $\mathcal{A}$. In terms of Proposition IV.1.4.6 of \cite{Blackadar}, the double dual $\mathcal{A}^{\ast\ast}$ of $\A$ is a type ${\rm I}_{n}$ von Neumann algebra. By Proposition 1.21.13 of \cite{Sakai}, the identity mapping $id$ extends uniquely to a $\sigma(\mathcal{A}^{\ast\ast},\mathcal{A}^{\ast})$-{\scriptsize WOT} continuous $\ast$-homomorphism $\tilde{id}$ of $\mathcal{A}^{\ast\ast}$ onto $W^{\ast}(\mathcal{A})$. It follows that $W^{\ast}(\mathcal{A})$ is a type ${\rm I}_{n}$ von Neumann algebra. This completes the proof of this claim.
	
		\vspace{0.2cm}
	
	\noindent(\textit{End of the proof of Lemma }\ref{Tool-1.5}.) Without loss of generality, we assume that $\{e_{ij}\}_{i,j=1}^n$ is a system of matrix units for $W^*(\mathcal A)$. Let
	\begin{equation}\label{e}
		e:= \bigvee_{i,j=1}^n  R(e_{ij}a). 
	\end{equation}
	We can verify directly that $e_{ij}e=ee_{ij}e$ for all $1\le i,j\le n$. Thus, $e$ is a central projection  of $W^{\ast}(\mathcal{A})$. Since $a\in \mathcal{F}(\mathcal{M},\tau) $, the definition of $e$ in (\ref{e}) entails $e\in\mathcal{F}(\mathcal{M},\tau)$. As $a=\sum_{i=1}^n e_{ii}a$, we have that $R(a)\le e$. From the fact that each $e_{ij}a$ belongs to $W^{\ast}(\mathcal{A}\cap\mathcal{F}(\mathcal{M},\tau))$, we further conclude that $W^{\ast}(\mathcal{A})e\subseteq W^{\ast}(\mathcal{A}\cap\mathcal{F}(\mathcal{M},\tau))$.
\end{proof}

\begin{lemma}\label{Tool-1.5.1}
	Let $\mathcal{M}$ be a von Neumann factor with a faithful, normal, semifinite, tracial weight $\tau$. Suppose that $\mathcal{A}$ is a ${\rm C}^{\ast}$-subalgebra of $(\mathcal{M}, \tau)$, $\ast$-isomorphic to $M_n(\mathbb{C})$. Let $\phi$ and $\psi$ be $\ast$-homomorphisms of $\mathcal{A}$ into $\mathcal{M}$ such that
	\begin{equation}\label{trace-preserving}
		\tau\big(\phi(a)\big)=\tau\big(\psi(a)\big),\quad\forall \ a\in\mathcal{A}_{+}. 
	\end{equation}
	Then, there exists a partial isometry $v$ in $\mathcal{M}$ such that
	\begin{equation}\label{unitary-implement-FD-alg}
		\phi(a)=v^*\psi(a)v,\quad\forall \ a\in\mathcal{A}. 
	\end{equation}
	Moreover, if $\tau\big(\phi(a)\big)=\tau\big(\psi(a)\big)<\infty$, for each $a\in\mathcal{A}_{+}$, then there is a unitary operator $u$ in $\mathcal{M}$ such that $\phi(a)=u^*\psi(a)u$, for every $a\in\mathcal{A}$.
\end{lemma}

\begin{proof}
	Let $\{e_{ij}\}_{1\leq i,j\leq n}$ be a system of matrix units for $\mathcal{A}$. Then $\{\phi(e_{ij})\}_{1\leq i,j\leq n}$ is a system of matrix units for $\phi(\mathcal{A})$. So is $\{\psi(e_{ij})\}_{1\leq i,j\leq n}$ for $\psi(\mathcal{A})$. 
	
	Note that
	\begin{equation*}
		\tau(\phi(e_{ii}))=\tau(\psi(e_{ii})), \quad\forall \ 1\leq i\leq n. 
	\end{equation*}
	Since $\mathcal{M}$ is a factor, there exists a partial isometry $v_{i}$ in $\mathcal{M}$ such that
	\begin{equation*}
		\phi(e_{ii})=v^*_{i}v_{i}\quad\text{ and }\quad \psi(e_{ii})=v_{i}v^*_{i}. 
	\end{equation*}
	Let $v$ be defined as
	\begin{equation*}
		v:=\sum\nolimits_{1\leq i\leq n}\phi(e_{i1})v^*_{1}\psi(e_{1i}). 
	\end{equation*}
	Then, it is routine to verify that
	\begin{enumerate}
		\item\quad $v^*v=\psi(I_{\mathcal{A}})$\quad and\quad $vv^*=\phi(I_{\mathcal{A}})$;
		\item\quad $\phi(e_{ij})v=v\psi(e_{ij})$ for all $1\leq i,j\leq n$.
	\end{enumerate}
	This completes the proof of (\ref{unitary-implement-FD-alg}). 
	
	Furthermore, if $\tau\big(\phi(a)\big)=\tau\big(\psi(a)\big)<\infty$, for each $a\in\mathcal{A}_{+}$, then there exists a partial isometry $w$ in $\M$ such that $w^*w=I-\psi(I_{\A})$ and $ww^*=I-\phi(I_{\A})$. Define $u=v+w$. It follows that $u$ is a unitary operator in $\M$ as desired.
\end{proof}

\begin{lemma}\label{Tool-1.5.3}
	Let $\mathcal{M}$ be a von Neumann factor with a faithful, normal, semifinite, tracial weight $\tau$. Suppose that $\mathcal{A}$ is a (separable) locally homogeneous ${\rm C}^{\ast}$-subalgebra of $(\mathcal{M}, \tau)$. Let $\mathcal{F}$ be a finite subset of $\mathcal{A}$ and $e$ be a projection in $\mathcal{A}^{\ast\ast}$ satisfying $ae=ea$ for each $a$ in $\mathcal{F}$, where $\A$ is identified with its imbedded image  in $\mathcal{A}^{\ast\ast}$.
	
	For each $\epsilon>0$, there is a finite dimensional von Neumann subalgebra $\mathcal{B}$ of $\mathcal{A}^{\ast\ast}$ with  $e$ being the identity such that, for each $a$ in $\mathcal{F}$, there exists an element $b$ in $\mathcal{B}$ satisfying
	\begin{equation*}
		\Vert ae-b\Vert<\epsilon. 
	\end{equation*}
\end{lemma}

\begin{proof}
	By applying IV.1.4.6 of \cite{Blackadar}, A ${\rm C}^{\ast}$-algebra $\mathcal{A}$ is $n$-homogeneous if and only if the double dual $\mathcal{A}^{\ast\ast}$ of $\mathcal{A}$ is a type ${\rm I}_n$ von Neumann algebra. By taking a finite partition in the center of $\mathcal{A}^{\ast\ast}$ fine enough, we can complete the proof.
\end{proof}

\begin{definition}\label{id&rho}
	Let $\mathcal{M}$ be a von Neumann algebra with a faithful, normal, semifinite, tracial weight $\tau$. Suppose that $\mathcal{A}$ is a separable ${\rm C}^{\ast}$-subalgebra of $\mathcal{M}$.  
	
	By virtue of Lemma $\ref{Tool-1.4}$, the projection $p_{\mathcal{K}(\mathcal{A},\tau)}$ reduces $\mathcal{A}$. Define
	\begin{equation}\label{id_{0}&id_{e}}
		id_{0}(a):=ap_{\mathcal{K}(\mathcal{A},\tau)} \quad \mbox{ and } \quad id_{e}(a):=ap^{\perp}_{\mathcal{K}(\mathcal{A},\tau)} \quad \forall \ a\in \mathcal{A}. 
	\end{equation}
	Then $id_{0}$ and $id_{e}$ are well-defined $\ast$-homomorphisms of $\mathcal{A}$ into $\mathcal{A}p_{\mathcal{K}(\mathcal{A},\tau)}$ and $\mathcal{A}p^{\perp}_{\mathcal{K}(\mathcal{A},\tau)}$, respectively.
	
	Let $\rho$ be a unital $\ast$-isomorphism of $\mathcal{A}$ into $\mathcal{M}$. Define
	\begin{equation}\label{rho_{0}&rho_{e}}
		\rho_{0}(a):=id_{0}(\rho(a)) \quad \mbox{ and } \quad \rho_{e}(a):=id_{e}(\rho(a)) \quad \forall \ a\in \mathcal{A}.
	\end{equation}
	Then $\rho_{0}$ and $\rho_{e}$ are well-defined $\ast$-homomorphisms of $\mathcal{A}$ into $\rho(\mathcal{A})p_{\mathcal{K}(\rho(\mathcal{A}),\tau)}$ and $\rho(\mathcal{A})p^{\perp}_{\mathcal{K}(\rho(\mathcal{A}),\tau)}$, respectively.
\end{definition}

\begin{theorem}\label{Tool-1.6}
	Let $\mathcal{M}$ be a countably decomposable, properly infinite, semifinite von Neumann factor with a faithful normal semifinite tracial weight $\tau$. Suppose that $\mathcal{A}$ is a separable AH subalgebra of $(\mathcal{M}, \tau)$. Let $id$ and $\rho$ be unital $\ast$-homomorphisms of $\mathcal{A}$ into $\mathcal{M}$ such that $id$ and $\rho$ are approximately unitarily equivalent.
	
	Then $id_{0}$ and $\rho_{0}$, as in Definition $\ref{id&rho}$, are strongly-approximately-unitarily-equivalent over $\mathcal{A}$, as in Definition $\ref{AEC}$, i.e.
	\begin{equation}\label{id_0-rho_0}
		id_{0}\sim_{\mathcal{A}}\rho_{0}\mod\ \ \mathcal{K}(\mathcal{M},\tau).
	\end{equation}
\end{theorem}

\begin{proof}
	Since $\mathcal{A}$ is AH, as in Remark \ref{AH}, it is convenient to assume that there is a monotone increasing sequence of locally homogeneous ${\rm C}^{\ast}$-subalgebras $\{\mathcal{A}_{n}\}_{n\geq 1}$ of $(\mathcal{M}, \tau)$ such that
	\begin{equation*}
		\mathcal{A}=\overline{\cup^{}_{n\geq 1}\mathcal{A}_{n}}^{\Vert\cdot\Vert}. 
	\end{equation*}
	
	We assume that $\mathcal{A}\cap\mathcal{K}(\mathcal{M},\tau)\neq \mathbf{0}$.
	
	Let $\{x_{n}\}^{}_{n\geq 1}$ be a sequence of positive, finite $\M$-rank operators in the unit ball of $\cup^{}_{n\geq 1}\mathcal{A}_{n}$, which is $\Vert\cdot\Vert$-norm dense in the unit ball of $\mathcal{A}_{+}\cap\mathcal{K}(\mathcal{M},\tau)$. Define two projections $p$ and $q$ as follows
	\begin{equation}\label{P1-and-Q1}
		p:=\vee^{}_{n\geq 1}R(x_{n})\quad\mbox{ and }\quad q:=\vee_{n\ge 1} R(\rho(x_{n})).
	\end{equation}
	Then, by virtue of Lemma \ref{Tool-rep-C^star-alg} and Lemma \ref{Tool-1.4}, we obtain that 
	\begin{equation*}
		p=p_{\mathcal{K}(\mathcal{A},\tau)} \quad \mbox{ and }\quad q=p_{\mathcal{K}(\rho(\mathcal{A}),\tau)}. 
	\end{equation*}

	Let $\mathcal{F}_{1}\subseteq\mathcal{F}_{2}\subseteq\cdots$ be a monotone increasing sequence of finite subsets of the unit ball of $\cup^{}_{n\geq 1}\mathcal{A}_{n}$ such that $\cup^{}_{n\geq 1}\mathcal{F}_{n}$ is $\Vert\cdot\Vert$-norm dense in the unit ball of $\mathcal{A}$.
	
	In the following, we construct at most countably many, mutually orthogonal, finite $\M$-rank projections in $W^{\ast}(\mathcal{A})$ with nice properties.
	
	Choose $n_1\geq 1$ such that $\mathcal{F}_1\cup\{x_{1}\}\subset\mathcal{A}_{n_1}$. Since $\mathcal{A}_{n_1}$ is locally homogeneous, by virtue of Lemma \ref{Tool-1.5}, there is a central projection $p_1$ in $W^{\ast}(\mathcal{A}_{n_1})$ such that
	\begin{enumerate}
		\item $p_1$ belongs to $\mathcal{F}(\mathcal{M},\tau)$;
		\item $R(x_1)\leq p_1$ in $\mathcal{M}$;
		\item $W^{\ast}(\mathcal{A}_{n_1})p_1\subseteq W^{\ast}(\mathcal{A}_{n_1}\cap\mathcal{F}(\mathcal{M},\tau))$.
	\end{enumerate}

	Define $y_1:=x_1$. If $p_1=p$, then we complete the construction. Otherwise, suppose $k\geq 2$ and we obtain $y_1,\ldots,y_k$ in $\{x_{n}\}^{}_{n\geq 1}$ and $p_1,\ldots,p_k$ in $\mathcal{F}(\mathcal{M},\tau)$ satisfying
	\begin{enumerate}
		\item $y_{i+1}$ is the first element after $y_{i}$ in $\{x_{n}\}^{}_{n\geq 1}$ such that $(p-p_i)y_{i+1}\neq 0$, for $1\leq i \leq k-1$;
		\item $\mathcal{F}_{i+1} \cup\{y_{1},\ldots,y_{i+1}\} \subset \mathcal{A}_{n_{i+1}}$, for all $1\leq i\leq k-1$;
		\item $\mathcal{F}_{i+1} \cup\{y_{1},\ldots,y_{i+1},p_1,\ldots,p_i\} \subset W^{\ast}(\mathcal{A}_{n_{i+1}})$ and $n_i\leq n_{i+1}$, for all $1\leq i\leq k-1$;
		\item $p_{i+1}$ is a central projection in $W^{\ast}(\mathcal{A}_{n_{i+1}})$ such that
		\begin{equation*}
			p_i\vee R(y_{i+1})\leq p_{i+1}, \quad 1\leq i\leq k-1. 
		\end{equation*}
	\end{enumerate}
	If $p_{k}=p$, then we complete the construction. Otherwise, let $y_{k+1}$ be the first element after $y_{k}$ in $\{x_{n}\}^{}_{n\geq 1}$ such that $(p-p_k)y_{k+1}\neq 0$. Choose $n_{k+1}\geq n_k$ such that
	\begin{equation*}
		\mathcal{F}_{k+1}\cup\{y_{1},\ldots,y_{k+1}\}\subset\mathcal{A}_{n_{k+1}}. 
	\end{equation*}
	Note that the projections $p_1,\ldots,p_k$ are also in $W^{\ast}(\mathcal{A}_{n_{k+1}})$. In terms of Lemma \ref{Tool-1.5}, there is a central projection $p_{k+1}$ of $W^{\ast}(\mathcal{A}_{n_{k+1}})$ such that:
	\begin{enumerate}
		\item $p_{k+1}$ belongs to $\mathcal{F}(\mathcal{M},\tau)$;
		\item $p_k\vee R(y_{k+1})\leq p_{k+1}$ in $W^{\ast}(\mathcal{A}_{n_{k+1}})$;
		\item $W^{\ast}(\mathcal{A}_{n_{k+1}})p_{k+1}\subseteq W^{\ast}(\mathcal{A}_{n_{k+1}}\cap\mathcal{F}(\mathcal{M},\tau))$.
	\end{enumerate}
	Define $e_{1}:=p_{1}$ and $e_{k+1}:=p_{k+1}-p_{k}$ for each $k \geq 1$. Let $\mathcal{F}_{0}:=\mathcal{F}_{1}$. It follows that $ae_i=e_ia$ for each $a$ in $\mathcal{F}_{j}$ and $i=j+1,\ldots,k+1$, where $j\geq 0$.
	
	Recursively, we obtain a sequence of at most countably many, mutually orthogonal, $(\mathcal{M},\tau)$-finite-rank projections $\{e_{i}\}_{1\leq i\leq N}$ in $\mathcal{F}(\mathcal{M},\tau)$ such that
	\begin{equation*}
		\mbox{\scriptsize SOT-}\sum\nolimits_{1\leq i \leq N}e_{i}=p, 
	\end{equation*}
	where $N\in \mathbb{N}\cup\{\infty\}$.

	By applying Lemma \ref{Tool-rep-C^star-alg}, the unital $\ast$-homomorphism $\rho$ extends uniquely to a normal $\ast$-isomorphism $\rho'$ of the {\scriptsize WOT}-closure of $\mathcal{A}\cap\mathcal{F}(\mathcal{M},\tau)$ into $\mathcal{M}$ such that
	\begin{equation*}
		\tau(\rho'(a))=\tau(a),\quad  \forall \ a \in  W^{\ast}\big(\mathcal{A}\cap\mathcal{F}(\mathcal{M},\tau)\big)_{+}. 
	\end{equation*}
	Moreover, by the preceding arguments and Lemma \ref{Tool-1.4}, we have
	\begin{equation*}
		\mbox{\scriptsize SOT-}\sum\nolimits_{1\leq i \leq N}\rho^{\prime}(e_{i})=q. 
	\end{equation*}
	and, for each projection $e$ in $W^{\ast}(\mathcal{A}_{n_{i}})p_{i}$, we have
	\begin{equation*}
		\tau(e)=\tau(\rho^{\prime}(e))<\infty . 
	\end{equation*}

	Fix $\epsilon>0$. For each $i\geq 1$, in terms of Lemma \ref{Tool-1.5.3}, there exists a finite dimensional von Neumann algebra $\mathcal{B}_{i}$ containing $e_{i}$ as its identity, in $W^{\ast}(\mathcal{A}_{n_{i}})p_{i}$, such that, for each $a\in\mathcal{F}_{i-1}$, there is an operator $a_{i}\in\mathcal{B}_{i}$ satisfying
	\begin{equation*}
		\Vert a_{}e_i-a_{i}\Vert<\frac{\epsilon}{2^{i+1}}. 
	\end{equation*}
	Then, by virtue of Lemma \ref{Tool-1.5.1}, we obtain a partial isometry $u_{i}$ in $\mathcal{M}$, for $1\leq i\leq N$, such that
	\begin{equation*}
		e_{i}=u^{\ast}_{i}u^{}_{i},\qquad \rho^{\prime}(e_{i})=u^{}_{i}u^{\ast}_{i}, \quad  \mbox{ and } \quad u^{}_{i}bu^{\ast}_{i}=\rho^{\prime}(b),\quad \forall \ b\in\mathcal{B}_{i}.
	\end{equation*}
	It follows that, for each $a\in\mathcal{F}_{i-1}$ and $i\geq 1$, 
	\begin{equation}\label{ineq-id-rho}
		\Vert u^{}_{i}(a_{}e_{i})u^{*}_{i}-\rho^{\prime}(a_{}e_{i})\Vert\leq \Vert u^{}_{i}(a_{}e_{i}-a_i)u^{*}_{i}\Vert+\Vert \rho^{\prime}(a_i-a_{}e_{i})\Vert<\frac{\epsilon}{2^{i}}.
	\end{equation}

	Define $u_{}={\sum}_{1\leq i \leq N}u_{i}$ in $\mathcal{M}$. It follows that
	\begin{equation}\label{partial-iso}
		u^{\ast}_{}u^{}_{}=p \quad\mbox{ and }\quad q=uu^{\ast}_{}.
	\end{equation}
	Note that for each $a$ in $\cup_{i\geq 1}\mathcal{F}$,
	\begin{equation*}
		id_{0}(a)=\mbox{\scriptsize SOT-}\sum\nolimits_{1\leq i \leq N}ae_{i}\quad \mbox{ and }\quad  \rho_{0}(a)=\mbox{\scriptsize SOT-}\sum\nolimits_{1\leq i \leq N}\rho^{\prime}(ae^{}_{i}). 
	\end{equation*}
	By (\ref{ineq-id-rho}) and (\ref{partial-iso}), we have that:
	\begin{enumerate}
		\item for every $a$ in $\mathcal{F}_{0}$,
		 \begin{equation*}
		 \begin{aligned}
		 	&\Vert u^{}_{}id_{0}(a)u^{\ast}_{}-\rho_{0}(a)\Vert=\Vert ({\sum}_{1\leq i \leq N}u_{i})id_{0}(a)({\sum}_{1\leq i \leq N}u_{i})^{\ast}-\rho_{0}(a)\Vert; \\
		 	= &\Vert \sum\nolimits_{1\leq i,k,l\leq N }u_{l}e_{i}ae_{i}u^{\ast}_{k}-\rho_{0}(a) \Vert=\Vert \sum\nolimits_{1\leq i\leq N }\Big(u_{i}e_{i}ae_{i}u^{\ast}_{i}-\rho^{\prime}(ae_{i})\Big) \Vert \\
		 	\leq  &\sum\nolimits_{1\leq i\leq N }\Vert u_{i}e_{i}ae_{i}u^{\ast}_{i}-\rho^{\prime}(ae_{i}) \Vert< \epsilon.
		 \end{aligned}
		 \end{equation*}
		\item for every $a$ in $\cup_{i\geq 1}\mathcal{F}_{i}$, a similar computation implies that
		 \begin{equation*}
		 	\Vert u^{}_{}id_{0}(a)u^{\ast}_{}-\rho_{0}(a)\Vert<\infty\quad\mbox{ and }\quad u^{}_{}id_{0}(a)u^{\ast}_{}-\rho_{0}(a)\in\mathcal{K}(\mathcal{M},\tau).
		 \end{equation*}
	\end{enumerate}
	For each $j\geq 1$, define $\{\mathcal{E}_{i}=\mathcal{F}_{i+j-1}\}_{i\geq 1}$ and $\mathcal{E}_{0}:=\mathcal{E}_{1}$. We can iterate the preceding arguments to construct a partial isometry $v_{j}$  in $\mathcal{M}$ with respect to $\{\mathcal{E}_{i}\}_{i\geq 0}$ such that
	\begin{enumerate}
		\item for every $a$ in $\cup_{i\geq 1}\mathcal{F}_{i}$ and $j\geq 1$,
		 \begin{equation*}
		 	\Vert v_{j}id_{0}(a)v^{\ast}_{j}-\rho_{0}(a)\Vert<\infty\quad\mbox{ and }\quad v_{j}id_{0}(a)v^{\ast}_{j}-\rho_{0}(a)\in\mathcal{K}(\mathcal{M},\tau);
		 \end{equation*}
		\item for each $a$ in $\mathcal{F}_{j}$,
		 \begin{equation*}
		 	\Vert v_{j}id_{0}(a)v^{\ast}_{j}-\rho_{0}(a)\Vert<\frac{1}{2^{j}}.
		 \end{equation*}
	\end{enumerate}
	Note that $\cup_{i\geq 1}\mathcal{F}_i$ is $\Vert\cdot\Vert$-norm dense in the unit ball of $\mathcal{A}$. Thus, for each $a$ in $\mathcal{A}$, we obtain that
	\begin{equation*}
		 	\lim\nolimits_{j\rightarrow\infty}\Vert v^{\ast}_{j}id_{0}(a)v_{j}-\rho_{0}(a)\Vert=0.
	\end{equation*}
	This completes the proof of (\ref{id_0-rho_0}).
\end{proof}

We cite Theorem 5.3.1 of \cite{Li} as another important tool. Note that, in the remainder, the symbol ``$\sim_{\mathcal{A}}$'' follows from Definition \ref{AEC}.

\begin{theorem}\label{VoiThmNuclear}
	Let $\mathcal{M}$ be a countably decomposable, properly infinite, semifinite factor with a faithful, normal, semifinite, tracial weight $\tau$. Let $\mathcal{K}(\mathcal{M},\tau)$ be the set of compact operators in $(\mathcal{M},\tau)$. Suppose that $\mathcal{A}$ is a separable nuclear $C^*$-subalgebra of $\mathcal{M}$ with an identity $I_{\mathcal{A}}$. If $\rho:\mathcal{A}\rightarrow\mathcal{M}$ is a $*$-homomorphism satisfying $\displaystyle \rho(\mathcal{A}\cap \mathcal{K}(\mathcal{M},\tau))={\mathbf 0}$, then
	\begin{equation*}
		id_{\mathcal{A}} \sim_{\mathcal{A}} id_{\mathcal{A}} \oplus \rho \quad \mod\ \ \mathcal{K}(\mathcal{M},\tau). 
	\end{equation*}
\end{theorem}

We are ready for the main theorem.

\begin{theorem}\label{AH-main-thm}
	Let $\mathcal{M}$ be a countably decomposable, infinite, semifinite factor with a faithful normal semifinite tracial weight $\tau$. Suppose that $\mathcal{A}$ is a separable AH subalgebra of $\mathcal{M}$ with an identity $I_{\mathcal{A}}$.

	If $\phi$ and $\psi$ are unital $\ast$-homomorphisms of $\mathcal{A}$ into $\mathcal{M}$, then the following statements are equivalent:
	\begin{enumerate}
		\item[(i)] $\phi \sim_{a}\psi$ \ in $\mathcal{M}$;
		\item[(ii)] $\phi \sim_{\mathcal{A}}\psi \mod\ \mathcal{K}(\mathcal{M},\tau)$.
	\end{enumerate}
\end{theorem}

\begin{proof}
	Note that the implication $(ii)\Rightarrow(i)$ is easy by Definition \ref{AEC}. Thus, we only need to prove the implication $(i)\Rightarrow(ii)$.

	The assumption $\phi \sim_{a}\psi\ \mbox{ in }\mathcal{M}$ entails that $\phi$ and $\psi$ have the same kernel. It follows that the mapping
	\begin{equation*}
		\rho:\phi(\mathcal{A})\rightarrow\psi(\mathcal{A}),\quad\text{ defined by }\quad\rho(b):=\psi(\phi^{-1}(b)),\ \forall\ b\in\phi(\mathcal{A}) 
	\end{equation*}
	is a well-defined $\ast$-isomorphism of $\phi(\mathcal{A})$ onto $\psi(\mathcal{A})$. Moreover, the following are equivalent:
	\begin{enumerate}
		\item\quad $\phi \sim_{\mathcal{A}}\psi \mod\ \mathcal{K}(\mathcal{M},\tau)$;
		\item\quad $id_{\phi(\mathcal{A})} \sim_{\phi(\mathcal{A})}\rho \mod\ \mathcal{K}(\mathcal{M},\tau)$.
	\end{enumerate}
	
	In terms of Lemma \ref{Tool-rep-C^star-alg}, the restriction of $\rho$ on $\phi(\mathcal{A})\cap\mathcal{F}(\mathcal{M},\tau)$ extends uniquely to a normal $\ast$-isomorphism of the {\scriptsize WOT}-closure of $\phi(\mathcal{A})\cap\mathcal{F}(\mathcal{M},\tau)$ into $\mathcal{M}$. Furthermore, Lemma \ref{Tool-rep-C^star-alg} and Lemma \ref{Tool-1.4} guarantee that there exists a sequence $\{x_{n}\}^{}_{n\geq 1}$ of positive, $(\mathcal{M},\tau)$-finite-rank operators in the unit ball of $\phi(\mathcal{A})_{+}\cap\mathcal{F}(\mathcal{M},\tau)$ such that the projections
	\begin{equation}\label{P2-and-Q2}
			p:=\vee^{}_{n\geq 1}R(x_{n})\quad\mbox{ and }\quad q:=\vee_{n\ge 1} R(\rho(x_{n}))
	\end{equation}
	reduce $\phi(\mathcal{A})$ and $\psi(\mathcal{A})$, respectively. Moreover, we have that
	\begin{equation*}
		px=x,\quad q\rho(x)=\rho(x), \quad \forall\, x\in \phi(\mathcal{A})\cap\mathcal{K}(\mathcal{M},\tau). 
	\end{equation*}
	Thus, the identity mapping $id$ on $\phi(\mathcal{A})$ can be expressed in the form
	\begin{equation}\label{eq-id-decomp}
		id=id_{0}\oplus id_{e},
	\end{equation}
	where $id_{0}$ is the compression of $id(\cdot) p$ on $\mbox{ran}\, p$, and  $id_{e}$ is the compression of $id(\cdot) p^{\perp}$ on $\mbox{ran}\, p^{\perp}$. We also write that
	\begin{equation*}
		id_{0}(\phi(\mathcal{A}))=\phi_0(\mathcal{A})\quad\mbox{ and }\quad id_{e}(\phi(\mathcal{A}))=\phi_e(\mathcal{A}). 
	\end{equation*}
	It follows that $id_{e}(\phi(\mathcal{A})\cap\mathcal{K}(\mathcal{M},\tau))={\mathbf 0}$.
	
	Likewise, the $\ast$-isomorphism $\rho$ of $\phi(\mathcal{A})$ can be expressed in the form
	\begin{equation}\label{eq-rho-decomp}
		\rho=\rho_{0}\oplus \rho_{e},
	\end{equation}
	where $\rho_{0}(A)=\rho(A)q|_{\text{ran} q}$ and $\rho_{e}(A)=\rho(A)q^{\perp}|_{\text{ran} q^{\perp}}$ for every $a$ in $\phi(\mathcal{A})$. We also write that
	\begin{equation*}
		\rho_{0}(\phi(\mathcal{A}))=\psi_0(\mathcal{A})\quad\mbox{ and }\quad\rho_{e}(\phi(\mathcal{A})) =\psi_e(\mathcal{A}).
	\end{equation*}
	It follows that $\rho_{e}(\phi(\mathcal{A})\cap\mathcal{K}(\mathcal{M},\tau))={\mathbf 0}$. 
	
	By virtue of  Theorem \ref{Tool-1.6}, there exists a partial isometry $w$ in $\mathcal{M}$ such that
	\begin{equation*}
		p=w^*w \quad \mbox{ and } \quad q=ww^*.
	\end{equation*}
		
	It is worth noting that,  in general,  many operators in $\phi_0(\mathcal{A})$ don't belong to $\phi(\mathcal{A})\cap\mathcal{K}(\mathcal{M},\tau)$. This is the motivation to develop Theorem \ref{Tool-1.6}.

	By virtue of (\ref{Induc-lim-def}), there exists a monotone increasing sequence $\mathcal{F}_{1}\subseteq\mathcal{F}_{2}\subseteq\cdots$ of finite subsets of the unit ball of $\cup^{}_{k\geq 1}\mathcal{A}_{k}$ such that $\cup_{k\geq 1}\mathcal{F}_{k}$ is $\Vert\cdot\Vert$-norm dense in the unit ball of $\mathcal{A}$. Likewise, the union $\cup_{k\geq 1}\phi(\mathcal{F}_{k})$ (resp. $\cup_{k\geq 1}\psi(\mathcal{F}_{k})$) is $\Vert\cdot\Vert$-norm dense in the unit ball of $\phi(\mathcal{A})$ (resp. $\psi(\mathcal{A})$). Similarly, $\cup_{k\geq 1}\phi_{0}(\mathcal{F}_{k})$ (resp. $\cup_{k\geq 1}\psi_{0}(\mathcal{F}_{k})$) is $\Vert\cdot\Vert$-norm dense in the unit ball of $\phi_{0}(\mathcal{A})$ (resp. $\psi_{0}(\mathcal{A})$).
	
 	By applying Theorem \ref{Tool-1.6}, for every $k\geq 1$, there exists a partial isometry $v_{k}$ in $(\mathcal{M},\tau)$ such that the inequality
	\begin{equation*}
		\Vert {v}_k \phi_{0}(a){v}^*_k-\psi_0(a)\Vert<\frac{1}{2^{k}}
	\end{equation*}
	holds for every $a$ in $\mathcal{F}_k$.
	
	Furthermore, for every $a$ in $\mathcal{A}$, we have that ${v}_k \phi_{0}(a){v}^*_k-\psi_0(a)$ belongs to the ideal $\mathcal{K}(\mathcal{M},\tau)$. Therefore, there exists a sequence $\{{v}_{k}\}_{k\geq 1}$ of partial isometries in $\mathcal{M}$ such that
	\begin{enumerate}
		\item $\lim_{k\rightarrow \infty}\Vert{v}_k \phi_{0}(a){v}^*_k-\psi_0(a)\Vert=0$, for every $a$ in $\mathcal{A}$;
		\item ${v}_k \phi_{0}(a){v}^*_k-\psi_0(a)$ belongs to $\mathcal{K}(\mathcal{M},\tau)$ for every $a$ in $\mathcal{A}$ and $k\geq 1$.
	\end{enumerate}
	Notice that
	\begin{equation*}
		id_{e}(\phi_{}(\mathcal{A})\cap \mathcal{K}(\mathcal{M},\tau))=\rho_{e}(\phi_{}(\mathcal{A})\cap \mathcal{K}(\mathcal{M},\tau))={\mathbf 0}.
	\end{equation*}
	Thus, by applying Theorem \ref{VoiThmNuclear}, Theorem \ref{Tool-1.6}, and the decompositions in (\ref{eq-id-decomp}) and (\ref{eq-rho-decomp}), it follows that
	\begin{equation*}
		\begin{aligned}
		\phi=(id_{0}\circ\phi_{})\oplus(id_{e}\circ\phi_{})&\sim_{\mathcal{A}} (id_{0}\circ\phi_{})\oplus(id_{e}\circ\phi_{})\oplus(\rho_{e}\circ\phi_{}) & \mod \mathcal{K}(\mathcal{M},\tau)\\
		&= \phi_{0} \oplus \phi_{e}\oplus\psi_{e}\\
		&\sim_{\mathcal{A}} \psi_{0} \oplus \psi_{e}\oplus\phi_{e} & \mod \mathcal{K}(\mathcal{M},\tau)\\
		&= (\rho_{0}\circ\phi_{})\oplus(\rho_{e}\circ\phi_{})\oplus(id_{e}\circ\phi_{})\\
		&= (\rho_{}\circ\phi_{})\oplus(id_{e}\circ\phi_{})\sim_{\mathcal{A}} (\rho_{}\circ\phi_{})=\psi & \mod \mathcal{K}(\mathcal{M},\tau)
		\end{aligned}
	\end{equation*}
	This completes the proof.
\end{proof}

\vspace{1cm}


\begin{thebibliography}{99}


\bibitem{Ave} William Arveson.  Notes on extensions of C$^*$-algebras. \emph{Duke Math. J.} \textbf{44} (1977), no. 2, 329--355.

\bibitem{Berg} David Berg. An extension of the Weyl-von Neumann theorem to normal operators. \emph{Trans. Amer. Math. Soc.} \textbf{160} (1971), 365--371.

\bibitem{Blackadar} Bruce Blackadar. \emph{Operator algebras. Theory of ${\rm C}^{\ast}$-algebras and von Neumann algebras.} Encyclopaedia of Mathematical Sciences, 122. Operator Algebras and Non-commutative Geometry, III. Springer-Verlag, Berlin, 2006.

\bibitem{Niu} Alin Ciuperca, Thierry Giordano, Ping Wong Ng  and  Zhuang Niu. Amenability and uniqueness. \emph{Adv. Math.} \textbf{240} (2013), 325--345.

\bibitem{Davidson} Kenneth Davidson. \emph{${\rm C}^{\ast}$-algebras by example.} Fields Institute Monographs, 6. American Mathematical Society, Providence, RI, 1996.

\bibitem{DavidsonNormal} Kenneth Davidson. Normal operators are diagonal plus Hilbert-Schmidt. \emph{J. Operator Theory} \textbf{20} (1988), no. 2, 241--249.

\bibitem{Hadwin2} Huiru Ding and Don Hadwin. Approximate equivalence in von Neumann algebras. {\em Sci. China Ser. A} \textbf{48} (2005), no. 2, 239--247.


\bibitem{Gong3} George Arthur Elliott, Guihua Gong, Liangqing Li. On the classification of simple inductive limit ${\rm C}^{\ast}$-algebras. II. The isomorphism theorem. \emph{Invent. Math.} \textbf{168} (2007), no. 2, 249--320. 

\bibitem{Gong1} Guihua Gong, Chunlan Jiang, Liangqing Li, Cornel Pasnicu. A reduction theorem for AH algebras with the ideal property. \emph{Int. Math. Res. Not.} (2018), no. 24, 7606--7641. 


\bibitem{Hadwin1} Donald Hadwin. Nonseparable approximate equivalence. \emph{Trans. Amer. Math. Soc.} \textbf{266} (1981), no. 1, 203--231.

\bibitem{Hadwin3} Donald Hadwin and Rui Shi. A note on the Voiculescu's theorem for normal operators in semifinite von Neumann algebras.  \emph{Oper. Matrices}, \textbf{12} (2018), no. 4, 1129--1144.

\bibitem{Halmos} Paul Halmos. Ten problems in Hilbert space. \emph{Bull. Amer. Math. Soc.} \textbf{76} (1970), 887--933.


\bibitem{Kadison1} Richard  Kadison  and  John   Ringrose. \emph{Fundamentals of the theory of operator algebras. Vol. I. Elementary theory.} Reprint of the 1983 original. Graduate Studies in Mathematics, 15. American Mathematical Society, Providence, RI, 1997.

\bibitem{Kadison2} Richard  Kadison  and  John   Ringrose. \emph{Fundamentals of the theory of operator algebras. Vol. II. Advanced theory.} Corrected reprint of the 1986 original. Graduate Studies in Mathematics, 16. American Mathematical Society, Providence, RI, 1997.

\bibitem{Kaftal} Victor Kaftal. On the theory of compact operators in von Neumann algebras. II. \emph{Pacific J. Math.} \textbf{79} (1978), no. 1, 129--137.

\bibitem{Li} Qihui Li, Junhao Shen, Rui Shi. A generalization of the Voiculescu theorem for normal operators in semifinite von Neumann algebras. arXiv:1706.09522 [math.OA].

\bibitem{Li2} Qihui Li, Junhao Shen, Rui Shi, Liguang Wang. Perturbations of self-adjoint operators in semifinite von Neumann algebras: Kato-Rosenblum theorem. \emph{J. Funct. Anal.} \textbf{275} (2018), no. 2, 259--287.

\bibitem{Lin1} Huaxin Lin, Homomorphisms from AH-algebras. \emph{J. Topol. Anal.} \textbf{9} (2017), no. 1, 67--125.

\bibitem{Lin2} Huaxin Lin, Locally AH algebras. \emph{Mem. Amer. Math. Soc.} \textbf{235} (2015), no. 1107.

\bibitem{Lin3} Huaxin Lin, The range of approximate unitary equivalence classes of homomorphisms from AH-algebras. \emph{Math. Z.} \textbf{263} (2009), no. 4, 903--922.

\bibitem{Murray} Francis Joseph Murray and John Von Neumann. On rings of operators. \emph{Ann. of Math.} (2) \textbf{37} (1936), no. 1, 116--229.

\bibitem{Von2} John von Neumann. Charakterisierung des Spektrums eines Integraloperators. {\em Actualits Sci. Indust.} \textbf{229}, Hermann, Paris, 1935.

\bibitem{Niu1} Zhuang Niu, Mean dimension and AH-algebras with diagonal maps. \emph{J. Funct. Anal.} \textbf{266} (2014), no. 8, 4938--4994.

\bibitem{Ror} Mikael R{\o}rdam, Flemming Larsen, and Niels Jakob Laustsen, \emph{An introduction to K-theory for ${\rm C}^{\ast}$-algebras.}  London Mathematical Society Student Texts, 49. Cambridge University Press, Cambridge, 2000.


\bibitem{Sakai} Shoichiro Sakai, \emph{${\rm C}^{\ast}$-algebras and ${\rm W}^{\ast}$-algebras.} Reprint of the 1971 edition. Classics in Mathematics. Springer-Verlag, Berlin, 1998.

\bibitem{Takesaki} Masamichi Takesaki. \emph{Theory of operator algebras. I. Reprint of the first $(1979)$ edition.} Encyclopaedia of Mathematical Sciences, 124. Operator Algebras and Non-commutative Geometry, 5. Springer-Verlag, Berlin, 2002.

\bibitem{Takesaki3} Masamichi Takesaki. \emph{Theory of operator algebras. III.} Encyclopaedia of Mathematical Sciences, 127. Operator Algebras and Non-commutative Geometry, 8. Springer-Verlag, Berlin, 2003.


\bibitem{Voi2} Dan Voiculescu. A non-commutative Weyl-von Neumann theorem. \emph{Rev. Roumaine Math. Pures Appl.} \textbf{21} (1976), no. 1, 97--113.

\bibitem{Voi} Dan Voiculescu. Some results on norm-ideal perturbations of Hilbert space operators. \emph{J. Operator Theory} \textbf{2} (1979), no. 1, 3--37.


\bibitem{Weyl} Hermann Weyl. {\"{U}ber beschr\"{a}nkte quadratische formen, deren differenz vollstetig ist.} Rend. Circ. Mat. Palermo \textbf{27} (1) (1909), 373--392.

\bibitem{Zaido} L\'{a}szl\'{o} Zsid\'{o}. The Weyl-von Neumann theorem in semifinite factors. \emph{J. Funct. Anal.} \textbf{18} (1975), 60--72.
\end{thebibliography}
\end{document}